\documentclass{article}%
\usepackage{amssymb}
\usepackage{amsfonts}
\usepackage{amsmath}
\usepackage{graphicx}%
\providecommand{\U}[1]{\protect\rule{.1in}{.1in}}
\newtheorem{theorem}{Theorem}[section]

\newtheorem{corollary}[theorem]{Corollary}

\newtheorem{definition}[theorem]{Definition}

\newtheorem{lemma}[theorem]{Lemma}

\newtheorem{proposition}[theorem]{Proposition}
\newtheorem{remark}[theorem]{Remark}

\newenvironment{proof}[1][Proof]{\textbf{#1.} }{\hfill\rule{0.5em}{0.5em}}
{\catcode`\@=11\global\let\AddToReset=\@addtoreset
\AddToReset{equation}{section}

\AddToReset{theorem}{section}

\begin{document}

\title{ Evolution equations of $p$-Laplace type with absorption or source terms and
measure data}
\author{Marie-Fran\c{c}oise BIDAUT-V\'{E}RON\thanks{Laboratoire de Math\'{e}matiques
et Physique Th\'{e}orique, CNRS UMR 7350, Facult\'{e} des Sciences, 37200
Tours France. E-mail: veronmf@univ-tours.fr}
\and Quoc-Hung NGUYEN\thanks{Laboratoire de Math\'{e}matiques et Physique
Th\'{e}orique, CNRS UMR 7350, Facult\'{e} des Sciences, 37200 Tours France.
E-mail: Hung.Nguyen-Quoc@lmpt.univ-tours.fr}}
\date{.}
\maketitle

\begin{abstract}
Let $\Omega$ be a bounded domain of $\mathbb{R}^{N}$, and $Q=\Omega
\times(0,T).$ We consider problems\textit{ }of the type
\[
\left\{
\begin{array}
[c]{l}%
{u_{t}}-{\Delta_{p}}u\pm\mathcal{G}(u)=\mu\qquad\text{in }Q,\\
{u}=0\qquad\text{on }\partial\Omega\times(0,T),\\
u(0)=u_{0}\qquad\text{in }\Omega,
\end{array}
\right.
\]
where ${\Delta_{p}}$ is the $p$-Laplacian, $\mu$ is a bounded Radon measure,
$u_{0}\in L^{1}(\Omega),$ and $\pm\mathcal{G}(u)$ is an absorption or a source
term$.$ In the model case $\mathcal{G}(u)=\pm\left\vert u\right\vert ^{q-1}u$
$(q>p-1),$ or $\mathcal{G}$ has an exponential type. We prove the existence of
renormalized solutions for any measure $\mu$ in the subcritical case, and give
sufficient conditions for existence in the general case, when $\mu$ is good in
time and satisfies suitable capacitary conditions.

\end{abstract}

\section{Introduction}

Let $\Omega$ be a bounded domain of $\mathbb{R}^{N}$, and $Q=\Omega
\times(0,T),$ $T>0.$ We consider the quasilinear parabolic problem
\begin{equation}
\left\{
\begin{array}
[c]{l}%
{u_{t}}-\mathcal{A}(u)\pm\mathcal{G}(u)=\mu\qquad\text{in }Q,\\
{u}=0\qquad\qquad\qquad\qquad\qquad\text{on }\partial\Omega\times(0,T),\\
u(0)=u_{0}\qquad\qquad\qquad\qquad\text{in }\Omega,
\end{array}
\right.  \label{pga}%
\end{equation}
where $\mu$ is a bounded Radon measure on $Q,$ $u_{0}\in L^{1}(\Omega).$ We
assume that $\mathcal{A}(u)=$div$(A(x,\nabla u))$ and $A$ is a
Carath\'{e}odory function on $\Omega\times\mathbb{R}^{N}$, such that, for
$a.e.$ $x\in\Omega,$ and any $\xi,\zeta\in\mathbb{R}^{N},$
\begin{equation}
A(x,\xi).\xi\geq\Lambda_{1}\left\vert \xi\right\vert ^{p},\qquad\left\vert
A(x,\xi)\right\vert \leq\Lambda_{2}\left\vert \xi\right\vert ^{p-1}%
,\qquad\Lambda_{1},\Lambda_{2}>0, \label{condi3}%
\end{equation}%
\begin{equation}
(A(x,\xi)-A(x,\zeta)).\left(  \xi-\zeta\right)  >0\text{ if }\xi\neq\zeta,
\label{condi4}%
\end{equation}
for $p>1;$ and $\mathcal{G}(u)=\mathcal{G}(x,t,u),$ where $(x,t,r)\mapsto
\mathcal{G}(x,t,r)$ is a Caratheodory function on $Q\times\mathbb{R}$ with
\begin{equation}
\mathcal{G}(x,t,r)r\geq0,\quad\text{for }a.e.(x,t)\in Q\quad\text{and any
}r\in\mathbb{R}. \label{gun}%
\end{equation}
The model problem is relative to the $p$-Laplace operator: $\mathcal{A}%
(u)=\Delta_{p}u=\text{div}(|\nabla u|^{p-2}\nabla u)$, and $\mathcal{G}$ has a
power-type $\mathcal{G}(u)=\pm\left\vert u\right\vert ^{q-1}u$ $(q>p-1),$ or
an exponential type. Our aim is to give sufficient conditions on the measure
$\mu$ in terms of capacity to obtain existence results. We denote by
$\mathcal{M}_{b}(\Omega)$ and $\mathcal{M}_{b}(Q)$ the sets of bounded Radon
measures on $\Omega$ and $Q$ respectively.\medskip

Next we make a brief survey of the main works on problem (\ref{pga}). First we
consider the case of an\textit{ absorption term}$:$%
\begin{equation}
\left\{
\begin{array}
[c]{l}%
{u_{t}}-\mathcal{A}(u)+\mathcal{G}(u)=\mu\qquad\text{in }Q,\\
{u}=0\qquad\qquad\qquad\qquad\qquad\text{on }\partial\Omega\times(0,T),\\
u(0)=u_{0}\qquad\qquad\qquad\qquad\text{in }\Omega.
\end{array}
\right.  \label{abp}%
\end{equation}
$\medskip$

\noindent For $p=2,$ $\mathcal{A}(u)=\Delta u$ and $\mathcal{G}(u)=|u|^{q-1}u$
$(q>1)$, the pionnier results concern the case $\mu=0$ and $u_{0}$ is a Dirac
mass in $\Omega$, see \cite{BrFr}: existence holds if and only if $q<(N+2)/N.$
Then optimal results are given in \cite{BaPi2}, for any $\mu\in\mathcal{M}%
_{b}({Q})$ and $u_{0}\in\mathcal{M}_{b}(\Omega)$. Here two capacities are
involved: the elliptic Bessel capacity $\mathrm{Cap}_{\mathbf{G}_{\alpha},s}$
defined, for $\alpha>0,s>1$ and any Borel set $E\subset\mathbb{R}^{N},$ by
\[
\mathrm{Cap}_{\mathbf{G}_{\alpha},s}(E)=\inf\{||\varphi||_{L^{s}%
(\mathbb{R}^{N})}^{s}:\varphi\in L^{s}(\mathbb{R}^{N}),\varphi\geq0\quad
G_{\alpha}\ast\varphi\geq1\text{ on }E\},
\]
where $\mathbf{G}_{\alpha}$ is the Bessel kernel of order $\alpha$; and the
capacity Cap$_{2,1,s}$ defined, for any compact set $K\subset\mathbb{R}^{N+1}$
by
\[
\mathrm{Cap}_{2,1,s}(K)=\inf\left\{  ||\varphi||_{W_{s}^{2,1}(\mathbb{R}%
^{N+1})}^{s}:\varphi\in\mathcal{S}(\mathbb{R}^{N+1}),\quad\varphi\geq1\text{
on a neighborhood of }K\right\}  ,
\]
and extended classically to Borel sets, where
\[
||\varphi||_{W_{s}^{2,1}(\mathbb{R}^{N+1})}=||\varphi||_{L^{s}(\mathbb{R}%
^{N+1})}+||\varphi_{t}||_{L^{s}(\mathbb{R}^{N+1})}+||\left\vert \nabla
\varphi\right\vert ||_{L^{s}(\mathbb{R}^{N+1})}+\sum\limits_{i,j=1,2,...,N}%
||\varphi_{x_{i}x_{j}}||_{L^{s}(\mathbb{R}^{N+1})}.
\]
In \cite{BaPi2}, Baras and Pierre proved that there exists a solution if and
only if $\mu$ does not charge the sets of $\mathrm{Cap}_{2,1,\frac{q}{q-1}}%
$-capacity zero and $u_{0}$ does not charge the sets of $\mathrm{Cap}%
_{\mathbf{G}_{\frac{2}{q}},\frac{q}{q-1}}$-capacity zero.

The case where $\mathcal{G}$ has an exponential type was initiated by
\cite{Fu}, and studied in the framework of Orlicz spaces in \cite{RuTe,Io},
and very recently by \cite{NPhu} in the context of Wolff parabolic
potentials.\medskip

For $p\neq2$, most of the contributions are relative to the case
$\mathcal{G}(u)=|u|^{q-1}u,$ $\mu=0,$ with $\Omega$ bounded, or $\Omega
=\mathbb{R}^{N}$. The case where $u_{0}$ is a Dirac mass in $\Omega$ was
studied in \cite{Gm,KaVa} when $p>2$, and \cite{ChQiWa} when $p<2$. Existence
and uniqueness hold in the subcritical case
\begin{equation}
q<p_{c}:=p-1+\frac{p}{N}. \label{pc}%
\end{equation}
If $q$ $\geq p_{c}$ and $q>1$, there is no solution with an isolated
singularity at $t=0$. For $q<p_{c},$ and $u_{0}\in\mathcal{M}_{b}^{+}%
(\Omega),$ the existence was obtained in the sense of distributions in
\cite{Zh}, and for any $u_{0}\in\mathcal{M}_{b}(\Omega)$ in \cite{BiChVe}. The
case $\mu\in$ $L^{1}(Q),$ $u_{0}=0$ was treated in \cite{DAOr}, and with
$\mu\in$ $L^{1}(Q),$ $u_{0}\in L^{1}(\Omega)$ in \cite{AndSbWi}, where
$\mathcal{G}$ can be multivalued. A larger set of measures, introduced in
\cite{DrPoPr}, was studied in \cite{PePoPor}. Let $\mathcal{M}_{0}(Q)$ be the
set of Radon measures $\mu$ on $Q$ that do not charge the sets of zero
$c_{p}^{Q}$-capacity, where for any Borel set $E\subset Q,$
\[
c_{p}^{Q}(E)=\inf(\inf_{E\subset U\text{ open}\subset Q}\{||u||_{W}:u\in
W,u\geq\chi_{U}\quad a.e.\text{ in }Q\}),
\]
and $W$ is the space of functions ${z\in{L^{p}}((0,T);W_{0}^{1,p}(\Omega
)\cap{L^{2}}(\Omega))}$ such that ${{z_{t}}\in{L^{p^{\prime}}}%
((0,T);W^{-1,p^{\prime}}(\Omega)}+{{L^{2}}(\Omega))}$ imbedded with the norm%
\[
\left\Vert z\right\Vert _{W}=\left\Vert z\right\Vert _{L{{^{p}}((0,T);W_{0}%
^{1,p}(\Omega)\cap{L^{2}}(\Omega))}}+\left\Vert z_{t}\right\Vert _{{{_{t}}%
\in{L^{p^{\prime}}}((0,T);W^{-1,p^{\prime}}(\Omega)}+{{L^{2}}(\Omega))}}.
\]
It was shown that existence and uniqueness hold for any measure $\mu
\in\mathcal{M}_{b}(Q)\cap\mathcal{M}_{0}(Q),$ called regular, or diffuse, and
$p>1$, and for any function $\mathcal{G}\in C(\mathbb{R)}$ such that
$\mathcal{G}(u)u\geq0.$ Up to our knowledge, up to now no existence results
have been obtained for a measure $\mu\not \in $ $\mathcal{M}_{0}(Q).$\medskip

The case of a source term%
\begin{equation}
\left\{
\begin{array}
[c]{l}%
{u_{t}}-\mathcal{A}(u)=\mathcal{G}(u)+\mu\qquad\text{in }Q,\\
{u}=0\qquad\qquad\qquad\qquad\qquad\text{on }\partial\Omega\times(0,T),\\
u(0)=u_{0}\qquad\qquad\qquad\qquad\text{in }\Omega,
\end{array}
\right.  \label{sor}%
\end{equation}
with $\mathcal{G}(u)=u^{q}$ with nonnegative $u$ and $\mu,u_{0}$ was treated
in \cite{BaPi1} for $p=2,$ giving optimal conditions for existence. As in the
absorption case the arguments of proofs cannot be extended to general $p.$

\section{Main results}

In Section \ref{RS}, we introduce the notion of renormalized solutions, called
R-solutions, of problem (\ref{pga}), and we recall at Theorem \ref{sta} the
stability result that we proved in \cite{BiNgQu} for the problem without
perturbation
\begin{equation}
\left\{
\begin{array}
[c]{l}%
{u_{t}}-\mathcal{A}(u)=\mu\qquad\text{in }Q,\\
{u}=0\qquad\qquad\qquad\qquad\qquad\text{on }\partial\Omega\times(0,T),\\
u(0)=u_{0}\qquad\qquad\qquad\qquad\text{in }\Omega.
\end{array}
\right.  \label{pmu}%
\end{equation}
under the assumption
\[
p>p_{1}:=(2N+1)/(N+1){,}%
\]
that \textit{we make in all the sequel}. This condition ensures that the
functions $u$ and $\left\vert \nabla u\right\vert $ are well defined in
$L^{1}(Q).$ Combined with some approximation properties of the measures,
Theorem \ref{sta} is the key point of our results. \medskip

In Section \ref{subc}, we first give existence results of subcritical type,
valid for any measure $\mu\in\mathcal{M}_{b}(Q).$ Let $G\in C(\mathbb{R}^{+})$
be a nondecreasing function with values in $\mathbb{R}^{+},$ such that
\begin{equation}
\left\vert \mathcal{G}(x,t,r)\right\vert \leq G(|r|)\quad\text{for }a.e.\text{
}x\in\Omega\text{ and any }~r\in\mathbb{R}, \label{isp}%
\end{equation}%
\begin{equation}
\int_{1}^{\infty}G(s)s^{-1-p_{c}}ds<\infty, \label{asg}%
\end{equation}
where $p_{c}$ is defined at (\ref{pc}).

\begin{theorem}
\label{newa} Assume (\ref{gun}), (\ref{isp}), (\ref{asg}). Then, for any
$\mu\in\mathcal{M}_{b}(Q)$ and $u_{0}\in L^{1}(\Omega),$ there exists a
R-solution u of problem
\begin{equation}
\left\{
\begin{array}
[c]{l}%
{u_{t}}-\mathcal{A}(u)+\mathcal{G}(u)=\mu\qquad\text{in }Q,\\
{u}=0\qquad\text{in }\partial\Omega\times(0,T),\\
u(0)=u_{0}\qquad\text{in }\Omega.
\end{array}
\right.  \label{pro0}%
\end{equation}

\end{theorem}

\begin{theorem}
\label{news}Assume (\ref{gun}), (\ref{isp}), (\ref{asg}). There exists
$\varepsilon>0$ such that, for any $\lambda>0$, any $\mu\in\mathcal{M}_{b}%
^{+}(Q)$ and any nonneagtive $u_{0}\in L^{1}(\Omega),$ if $\lambda
+\mu(Q)+||u_{0}||_{L^{1}(\Omega)}\leq\varepsilon$, then there exists a
nonnegative R-solution u of problem
\begin{equation}
\left\{
\begin{array}
[c]{l}%
{u_{t}}-\mathcal{A}(u)=\lambda\mathcal{G}(u)+\mu\qquad\text{in }Q,\\
{u}=0\qquad\text{in }\partial\Omega\times(0,T),\\
u(0)=u_{0}\qquad\text{in }\Omega,
\end{array}
\right.  \label{pro1}%
\end{equation}

\end{theorem}

In particular for any if $\mathcal{G}(u)=\left\vert u\right\vert ^{q-1}u,$
condition (\ref{asg}) is equivalent to the fact that $q$ is subcritical:
$0<q<p_{c},$ where $p_{c}$ is defined at (\ref{pc}). \medskip

\noindent\ Next we consider the general case, with no subcriticality
assumptions, when $\mathcal{G}$ \textit{is nondecreasing in} $u$, and
$\mathcal{G}$ has a power type, or an exponential type. For $\mathcal{G}%
(u)=\left\vert u\right\vert ^{q-1}u$ for $q\geq p_{c},$ and $p\neq2,$ up to
now \textit{the good capacities for solving the problem are not known}. In the
following, we search sufficient conditions on the measures $\mu$ and $u_{0}$
ensuring that there exists a solution.To our knowledge, the question of
finding necessary conditions for existence is still an open problem.\medskip

In the sequel we give sufficient conditions for existence for \textit{measures
that have a good behaviour in }$t,$ based on recent results of \cite{BiNQVe}
relative to the elliptic case. We recall the notion of (truncated) Wolff
potential: for any nonnegative measure $\omega\in\mathcal{M}^{+}%
(\mathbb{R}^{N})$ any $R>0,$ $x_{0}\in\mathbb{R}^{N},$
\begin{equation}
\mathbf{W}_{1,p}^{R}[\omega]\left(  x_{0}\right)  =\int_{0}^{R}\left(
r^{p-N}\omega(B(x_{0},r))\right)  ^{\frac{1}{p-1}}\frac{dr}{r}. \label{wop}%
\end{equation}
Any measure $\omega\in\mathcal{M}_{b}(\Omega)$ is identified with its
extension by $0$ to $\mathbb{R}^{N}.$ In case of absorption, we obtain the following:

\begin{theorem}
\label{main1} Let $p<N$, $q>p-1,$ $\mu\in\mathcal{M}_{b}(Q)$, $f\in L^{1}(Q)$
and $u_{0}\in L^{1}(\Omega)$. Assume that%
\begin{equation}
\left\vert \mu\right\vert \leq\omega\otimes F,\text{ \quad with }\omega
\in\mathcal{M}_{b}^{+}(\Omega),F\in L^{1}((0,T)),F\geq0. \label{hypmu}%
\end{equation}
If $\omega$ does not charge the sets of $\mathrm{Cap}_{\mathbf{G}_{p},\frac
{q}{q+1-p}}$-capacity zero, then there exists a R-solution $u$ of problem
\begin{equation}
\left\{
\begin{array}
[c]{l}%
u_{t}-\mathcal{A}(u)+|u|^{q-1}u=f+\mu\qquad\text{in }Q,\\
{u}=0\qquad\text{on }\partial\Omega\times(0,T),\\
u(0)=u_{0}\qquad\text{in }\Omega.
\end{array}
\right.  \label{mainprob1}%
\end{equation}

\end{theorem}

From \cite[Proposition 2.3]{BaPi2}, a measure $\omega\in\mathcal{M}_{b}%
(\Omega)$ does not charge the sets of $\text{Cap}_{\mathbf{G}_{2},\frac
{q}{q-1}}$-capacity zero if and only if $\omega\otimes\chi_{(0,T)}$ does not
charge the sets of $\text{Cap}_{2,1,\frac{q}{q-1}}$-capacity zero. Therefore,
when $\mathcal{A}(u)=\Delta u$ and $\mu=\omega\otimes\chi_{(0,T)}$, $u_{0}\in
L^{1}(\Omega),$ we find again the existence result of \cite{BaPi2}. Besides,
in view of \cite[Theorem 2.16]{DrPoPr}, there exists data $\mu\in
\mathcal{M}_{b}(Q)$ in Theorem \ref{main1} such that $\mu\notin\mathcal{M}%
_{0}(Q)$, see Remark \ref{mac}, thus our result is the first one of existence
for non diffuse measure. Otherwise our result can be extended to a more
general function $\mathcal{G},$ see Remark \ref{exten}.\medskip

We also consider a source term. Denoting by $D=\sup_{x,y\in\Omega}\left\vert
x-y\right\vert $ the diameter of $\Omega,$ we obtain the following:

\begin{theorem}
\label{120410} Let $p<N$, $q>p-1$. Let $\mu\in\mathcal{M}_{b}^{+}(Q),$ and
nonnegative $u_{0}\in L^{\infty}(\Omega)$. Assume that%
\[
\mu\leq\omega\otimes\chi_{(0,T)},\text{ \quad with }\omega\in\mathcal{M}%
_{b}^{+}(\Omega).
\]
Then there exist $\lambda_{0}$ and $b_{0},$ depending of $N,p,q,\Lambda
_{1},\Lambda_{2},D,$ such that, if
\begin{equation}
\omega(E)\leq\lambda_{0}\mathrm{Cap}_{\mathbf{G}_{p},\frac{q}{q+1-p}}%
(E),\quad\forall E\text{ compact set}\subset\mathbb{R}^{N},\qquad
\text{and\quad}||u_{0}||_{L^{\infty}(\Omega)}\leq b_{0}, \label{051120132}%
\end{equation}
there exists a nonnegative R-solution $u$ of problem
\begin{equation}
\left\{
\begin{array}
[c]{l}%
u_{t}-\mathcal{A}(u)=u^{q}+\mu\qquad\text{in }Q,\\
u=0\qquad\text{on }\partial\Omega\times(0,T),\\
u(0)=u_{0}\qquad\text{in }\Omega,
\end{array}
\right.  \label{pro3}%
\end{equation}
which satisfies, $a.e.$ in $Q,$
\begin{equation}
{u(x,t)}\leq C\mathbf{W}_{1,p}^{2D}[\omega](x)+2||u_{0}||_{L^{\infty}(\Omega
)}, \label{maw}%
\end{equation}
where $C=C(N,p,\Lambda_{1},\Lambda_{2})$.\medskip
\end{theorem}

In case where $\mathcal{G}$ is an exponential, we introduce the notion of
maximal fractional operator, defined for any $\eta\geq0,$ $R>0,$ $x_{0}%
\in\mathbb{R}^{N}$ by
\[
\mathbf{M}_{p,R}^{\eta}[\omega](x_{0})=\sup_{r\in\left(  0,R\right)  }%
\frac{\omega(B(x_{0},r))}{r^{rN-p}h_{\eta}(r)},\qquad\text{where }h_{\eta
}(r)=\inf((-\ln r)^{-\eta},(\ln2)^{-\eta})).
\]
In the case of absorption, we obtain the following:

\begin{theorem}
\label{expo} Let $p<N$ and $\tau>0,\beta>1,\mu\in\mathcal{M}_{b}(Q),$ $f\in
L^{1}(Q)$ and $u_{0}\in L^{1}(\Omega)$. Assume that
\[
\left\vert \mu\right\vert \leq\omega\otimes F,\text{ \quad with }\omega
\in\mathcal{M}_{b}^{+}(\Omega),\;F\in L^{1}((0,T)),F\geq0,
\]
and that one of the following assumptions is satisfied:

\noindent(i) $||F||_{L^{\infty}((0,T))}\leq1,$ and for some $M_{0}%
=M_{0}(N,p,\beta,\tau,\Lambda_{1},\Lambda_{2},D),$
\begin{equation}
||\mathbf{M}_{p,2D}^{\frac{{p-1}}{\beta^{\prime}}}[\omega]|{|_{{L^{\infty}%
}({\mathbb{R}^{N}})}}<M_{0}; \label{plou}%
\end{equation}
(ii) there exists $\beta_{0}>\beta$ such that $\mathbf{M}_{p,2D}^{\frac{{p-1}%
}{\beta_{0}^{\prime}}}[\omega]\in L^{\infty}(\mathbb{R}^{N}).$

\medskip Then there exists a R-solution to the problem
\[
\left\{
\begin{array}
[c]{l}%
u_{t}-\mathcal{A}(u)+(e^{\tau|u|^{\beta}}-1)\mathrm{sign}u=f+\mu\qquad\text{in
}Q,\\
u=0\;\qquad\text{on }\partial\Omega\times(0,T),\\
u(0)=u_{0}\qquad\text{in }\Omega.
\end{array}
\right.
\]

\end{theorem}

In the case of a source term, we obtain:

\begin{theorem}
\label{MTH1} Let $\tau>0,l\in\mathbb{N}$ and $\beta\geq1$ such that
$l\beta>p-1.$ We set
\begin{equation}
\mathcal{E}(s)=e^{s}-\sum\limits_{j=0}^{l-1}{\frac{{{s^{j}}}}{{j!}}}%
,\qquad\forall s\in\mathbb{R}. \label{ess}%
\end{equation}
Let $\mu\in\mathcal{M}_{b}^{+}(Q)$, such that
\[
\mu\leq\omega\otimes\chi_{(0,T)},\text{\quad with }\omega\in\mathcal{M}%
_{b}^{+}(\Omega).
\]
Then, there exist $b_{0}$ and $M_{0}$ depending on $N,p,\beta,\tau
,l,\Lambda_{1},\Lambda_{2},D,$ such that if
\[
||\mathbf{M}_{p,2D}^{\frac{{(p-1)(\beta-1)}}{\beta}}[\omega]|{|_{{L^{\infty}%
}({\mathbb{R}^{N}})}}\leq M_{0},\quad\text{and \quad}||u_{0}||_{{L^{\infty}%
}(\Omega)}\leq b_{0},
\]
the problem
\begin{equation}
\left\{
\begin{array}
[c]{l}%
u_{t}-\mathcal{A}(u)=\mathcal{E}(\tau u^{\beta})+\mu\qquad\text{in }Q,\\
u=0\qquad\text{on }\partial\Omega\times(0,T),\\
u(0)=u_{0}\qquad\text{in }\Omega,
\end{array}
\right.  \label{pro2}%
\end{equation}
admits a nonnegative R-solution $u$, which satisfies, $a.e.$ in $Q,$ for some
$C=C(N,p,\Lambda_{1},\Lambda_{2})$,
\begin{equation}
{u(x,t)}\leq C\mathbf{W}_{1,p}^{2D}[\omega](x)+2{b_{0}}. \label{1334}%
\end{equation}

\end{theorem}

\section{Renormalized solutions and stability theorem\label{RS}}

Here we recall the definition of renormalized solutions of the problem without
perturbation (\ref{pmu}), given in \cite{Pe08} for $p>p_{1}.\medskip$

\noindent Let $\mathcal{M}_{s}(Q)$ be the set of measures $\mu\in
\mathcal{M}_{b}(Q)$ with support on a set of zero $c_{p}^{Q}$-capacity, also
called \textit{singular}. Let $\mathcal{M}_{b}^{+}(Q),\mathcal{M}_{0}%
^{+}(Q),\mathcal{M}_{s}^{+}(Q)$ be the positive cones of $\mathcal{M}%
_{b}(Q),\mathcal{M}_{0}(Q),\mathcal{M}_{s}(Q).$

\noindent Recall that any measure $\mu\in\mathcal{M}_{b}(Q)$ can be written
(in a unique way) under the form
\[
\mu=\mu_{0}+\mu_{s},\text{ where }\mu_{0}\in\mathcal{M}_{0}(Q),\quad\mu
_{s}=\mu_{s}^{+}-\mu_{s}^{-},\quad\text{with }\mu_{s}^{+},\mu_{s}^{-}%
\in\mathcal{M}_{s}^{+}(Q).
\]
In turn $\mu_{0}\in$ $\mathcal{M}_{0}(Q)$ admits (at least) a decomposition
under the form%
\[
\mu_{0}=f-\operatorname{div}g+h_{t},\qquad f\in L^{1}(Q),\quad g\in
(L^{p^{\prime}}(Q))^{N},\quad h\in{L^{p}((0,T);W_{0}^{1,p}(\Omega))},
\]
see \cite{DrPoPr}; and we write $\mu_{0}=(f,g,h).\medskip$

\noindent We set $T_{k}(r)=\max\{\min\{r,k\},-k\},$ for any $k>0$ and
$r\in\mathbb{R}$. If $u$ is a measurable function defined and finite $a.e.$ in
$Q$, such that $T_{k}(u)\in L^{p}((0,T);W_{0}^{1,p}(\Omega))$ for any $k>0$,
there exists a measurable function $w$ from $Q$ into $\mathbb{R}^{N}$ such
that $\nabla T_{k}(u)=\chi_{|u|\leq k}w,$ $a.e.$ in $Q,$ and for any $k>0$. We
define the gradient $\nabla u$ of $u$ by $w=\nabla u$. \medskip

\begin{definition}
\label{defin}Let { }$u_{0}\in L^{1}(\Omega),$ $\mu=\mu_{0}+\mu_{s}%
\in\mathcal{M}_{b}(${$Q$}$)$. A measurable function $u$ is a renormalized
solution, called \textbf{\ R-solution} of (\ref{pmu}) if there exists a
decompostion $(f,g,h)$ of $\mu_{0}$ such that
\[
U=u-h\in L^{\sigma}(0,T;W_{0}^{1,\sigma}(\Omega)\cap L^{\infty}(0,T;L^{1}%
(\Omega)),\quad\forall\sigma\in\left[  1,m_{c}\right)  ;\qquad T_{k}(U)\in
L^{p}((0,T);W_{0}^{1,p}(\Omega)),\quad\forall k>0;
\]
and:\medskip

(i) for any $S\in W^{2,\infty}(\mathbb{R})$ such that $S^{\prime}$ has compact
support on $\mathbb{R}$, and $S(0)=0$,%
\[%
\begin{array}
[c]{c}%
-\int_{\Omega}S(u_{0})\varphi(0)dx-\int_{Q}{{\varphi_{t}}S(U)}dxdt+\int%
_{Q}{S^{\prime}(U)A(x,t,\nabla u).\nabla\varphi}dxdt\\
+\int_{Q}{S^{\prime\prime}(U)\varphi A(x,t,\nabla u).\nabla U}dxdt=\int%
_{Q}f{S^{\prime}(U)\varphi}dxdt+\int_{Q}g.\nabla({S^{\prime}(U)\varphi
)dxdt{,}}%
\end{array}
\]
for any $\varphi\in L^{p}((0,T);W_{0}^{1,p}(\Omega))\cap L^{\infty}(Q)$ such
that $\varphi_{t}\in L^{p^{\prime}}((0,T);W^{-1,p^{\prime}}(\Omega))+L^{1}(Q)$
and $\varphi(.,T)=0$;\medskip

(ii) for any $\phi\in C(\overline{{Q}}),$%
\[
\lim_{m\rightarrow\infty}\frac{1}{m}\int\limits_{\left\{  m\leq U<2m\right\}
}{\phi A(x,t,\nabla u).\nabla U}dxdt=\int_{Q}\phi d\mu_{s}^{+},
\]%
\[
\lim_{m\rightarrow\infty}\frac{1}{m}\int\limits_{\left\{  -m\geq
U>-2m\right\}  }{\phi A(x,t,\nabla u).\nabla U}dxdt=\int_{Q}\phi d\mu_{s}%
^{-}.
\]

\end{definition}

In the sequel we consider the problem (\ref{pga}) where $\mu\in\mathcal{M}%
_{b}(Q)$, $u_{0}\in L^{1}(\Omega)$. We say that $u$ is a R-solution of problem
(\ref{pga}) if $\mathcal{G}(u)\in L^{1}(Q)$ and $u$ is a R-solution of
(\ref{pmu}) with data $(\mu\mp\mathcal{G}(u),u_{0})$.\medskip

We recall some properties of R-solutions which we proved in \cite[Propositions
2.8,2.10 and Remark 2.9]{BiNgQu}:

\begin{proposition}
\label{aid} Let $\mu\in L^{1}(Q)$ and $u_{0}\in L^{1}(\Omega),$ and $u$ be the
(unique) R-solution of problem (\ref{pga}) with data $\mu$ and $u_{0}$. Then
\begin{equation}
\mathrm{meas}\left\{  {|u|>k}\right\}  \leq C(||u_{0}||_{L^{1}(\Omega)}%
+|\mu|(Q))^{\frac{p+N}{N}}{k^{-p_{c}},}\qquad\forall k>0, \label{4bh1810131}%
\end{equation}
for some $C=C(N,p,\Lambda_{1},\Lambda_{2})$.\medskip
\end{proposition}

\begin{proposition}
\label{4bhmun} Let $\{\mu_{n}\}$ $\subset$ $\mathcal{M}_{b}(${$Q$}$),$ and
$\{u_{0,n}\}\subset L^{1}(\Omega),$ with
\[
\sup_{n}\left\vert {{\mu_{n}}}\right\vert ({Q})<\infty,\quad\text{and\quad
}\sup_{n}||{{u_{0,n}}}||_{L^{1}(\Omega)}<\infty.
\]
Let $\{u_{n}\}$ be a sequence of R-solutions of (\ref{pga}) with data $\mu
_{n}=\mu_{n,0}+\mu_{n,s}$ and $u_{0,n},$ relative to a decomposition
$(f_{n},g_{n},h_{n})$ of $\mu_{n,0}$. Assume that $\{f_{n}\}$ is bounded in
$L^{1}(Q)$, $\{g_{n}\}$ bounded in $(L^{p^{\prime}}(Q))^{N}$ and $\{h_{n}\}$
converges in $L^{p}(0,T;W_{0}^{1,p}(\Omega))$.\medskip

\noindent Then, up to a subsequence, $\{u_{n}\}$ converges to a function $u$
a.e in $Q$ and in $L^{s}(Q)$ for any $s\in\lbrack1,m_{c})$. Moreover, if
$\{\mu_{n}\}$ is bounded in $L^{1}(Q)$, then $\{u_{n}\}$ converges to $u$ in
$L^{s}(0,T;W_{0}^{1,s}(\Omega))$ in $s\in\lbrack1,p-\frac{N}{N+1})$.\medskip
\end{proposition}

Our results are based on the \textit{stability theorem} that we obtained for
problem (\ref{pmu}) in \cite{BiNgQu}, extending the elliptic result of
\cite[Theorem 3.4]{DMOP} to the parabolic case. Note that it is valid under
more general assumptions on the operator $\mathcal{A},$ see \cite{BiNgQu}.
Recall that a sequence $\left\{  \mu_{n}\right\}  $ $\subset\mathcal{M}%
_{b}(Q)$ converges to $\mu\in\mathcal{M}_{b}(Q)$ in the \textit{narrow
topology} of measures if%
\[
\lim_{n\rightarrow\infty}\int_{Q}\varphi d\mu_{n}=\int_{Q}\varphi d\mu
\qquad\forall\varphi\in C(Q)\cap L^{\infty}(Q).
\]

\begin{theorem}
\label{sta} Let $p>p_{1},${ }$u_{0}\in L^{1}(\Omega)$, and
\[
\mu=f-\operatorname{div}g+h_{t}+\mu_{s}^{+}-\mu_{s}^{-}\in\mathcal{M}_{b}%
({Q}),
\]
with $f\in L^{1}(Q),g\in(L^{p^{\prime}}(Q))^{N},$ $h\in L^{p}((0,T);W_{0}%
^{1,p}(\Omega))$ and $\mu_{s}^{+},\mu_{s}^{-}\in\mathcal{M}_{s}^{+}(Q).$ Let
$u_{0,n}\in L^{1}(\Omega),$
\[
\mu_{n}=f_{n}-\operatorname{div}g_{n}+(h_{n})_{t}+\rho_{n}-\eta_{n}%
\in\mathcal{M}_{b}({Q}),
\]
with \ $f_{n}\in L^{1}(Q),g_{n}\in(L^{p^{\prime}}(Q))^{N},h_{n}\in
L^{p}((0,T);W_{0}^{1,p}(\Omega)),$ and $\rho_{n},\eta_{n}\in\mathcal{M}%
_{b}^{+}({Q}),$ such that
\[
\rho_{n}=\rho_{n}^{1}-\operatorname{div}\rho_{n}^{2}+\rho_{n,s},\qquad\eta
_{n}=\eta_{n}^{1}-\mathrm{\operatorname{div}}\eta_{n}^{2}+\eta_{n,s},
\]
with $\rho_{n}^{1},\eta_{n}^{1}\in L^{1}(Q),\rho_{n}^{2},\eta_{n}^{2}%
\in(L^{p^{\prime}}(Q))^{N}$ and $\rho_{n,s},\eta_{n,s}\in\mathcal{M}_{s}%
^{+}(Q).$ Assume that \medskip%
\[
\sup_{n}\left\vert {{\mu_{n}}}\right\vert ({Q})<\infty,
\]
and $\left\{  u_{0,n}\right\}  $ converges to $u_{0}$ strongly in
$L^{1}(\Omega),$ $\left\{  f_{n}\right\}  $ converges to $f$ weakly in
$L^{1}(Q),$ $\left\{  g_{n}\right\}  $ converges to $g$ strongly in
$(L^{p^{\prime}}(Q))^{N}$, $\left\{  h_{n}\right\}  $ converges to $h$
strongly in $L^{p}((0,T);W_{0}^{1,p}(\Omega))$, $\left\{  \rho_{n}\right\}  $
converges to $\mu_{s}^{+}$ and $\left\{  \eta_{n}\right\}  $ converges to
$\mu_{s}^{-}$ in the narrow topology of measures; and $\left\{  \rho_{n}%
^{1}\right\}  ,\left\{  \eta_{n}^{1}\right\}  $ are bounded in $L^{1}(Q)$, and
$\left\{  \rho_{n}^{2}\right\}  ,\left\{  \eta_{n}^{2}\right\}  $ bounded in
$(L^{p^{\prime}}(Q))^{N}$.$\medskip$

Let $\left\{  u_{n}\right\}  $ be a sequence of R-solutions of
\[
\left\{
\begin{array}
[c]{l}%
{u_{n,t}}-\mathcal{A}(u_{n})=\mu_{n}\qquad\text{in }Q,\\
{u}_{n}=0\qquad\text{on }\partial\Omega\times(0,T),\\
u_{n}(0)=u_{0,n}\qquad\text{in }\Omega.
\end{array}
\right.
\]
relative to the decomposition $(f_{n}+\rho_{n}^{1}-\eta_{n}^{1},g_{n}+\rho
_{n}^{2}-\eta_{n}^{2},h_{n})$ of $\mu_{n,0}.$ Let $U_{n}=u_{n}-h_{n}.\medskip$

Then up to a subsequence, $\left\{  u_{n}\right\}  $ converges $a.e.$ in $Q$
to a R-solution $u$ of (\ref{pmu}), and $\left\{  U_{n}\right\}  $ converges
$a.e.$ in $Q$ to $U=u-h.$ Moreover, $\left\{  \nabla u_{n}\right\}  ,\left\{
\nabla U_{n}\right\}  $ converge respectively to $\nabla u,\nabla U$ $a.e.$ in
$Q,$ and $\left\{  T_{k}(U_{n})\right\}  $ converge to $T_{k}(U)$ strongly in
$L^{p}((0,T);W_{0}^{1,p}(\Omega))$ for any $k>0$.\medskip
\end{theorem}

For applying Theorem \ref{sta}, we require some approximation properties of
measures, see \cite{BiNgQu}:

\begin{proposition}
\label{4bhatt}Let $\mu=\mu_{0}+\mu_{s}\in\mathcal{M}_{b}^{+}(Q)$ with $\mu
_{0}\in\mathcal{M}_{0}^{+}(Q)$ and $\mu_{s}\in\mathcal{M}_{s}^{+}(Q).$

\noindent(i) Then, we can find a decomposition $\mu_{0}=(f,g,h)$ with $f\in
L^{1}(Q),g\in(L^{p^{\prime}}(Q))^{N},h\in L^{p}(0,T;W_{0}^{1,p}(\Omega))$ such
that
\begin{equation}
||f||_{L^{1}(Q)}+\left\Vert g\right\Vert _{(L^{p^{\prime}}(Q))^{N}%
}+||h||_{L^{p}(0,T;W_{0}^{1,p}(\Omega))}+\mu_{s}(\Omega)\leq2\mu(Q).
\label{4bhdeco}%
\end{equation}
(ii) Furthermore, there exists sequences of measures $\mu_{0,n}=(f_{n}%
,g_{n},h_{n})$ and $\mu_{s,n}$ such that $f_{n},g_{n},h_{n}\in C_{c}^{\infty
}(Q)$ strongly converge to $f,g,h$ in $L^{1}(Q),(L^{p^{\prime}}(Q))^{N}$ and
$L^{p}(0,T;W_{0}^{1,p}(\Omega))$ respectively, and $\mu_{s,n}\in(C_{c}%
^{\infty}(Q))^{+}$ converges to $\mu_{s}$ and $\mu_{n}:=\mu_{0,n}+\mu_{s,n}$
converges to $\mu$ in the narrow topology of measures, and satisfying
$|\mu_{n}|(Q)\leq\mu(Q),$
\begin{equation}
||f_{n}||_{L^{1}(Q)}+\left\Vert g_{n}\right\Vert _{(L^{p^{\prime}}(Q))^{N}%
}+||h_{n}||_{L^{p}(0,T;W_{0}^{1,p}(\Omega))}+\mu_{s,n}(Q)\leq2\mu(Q).
\label{4bhdecn}%
\end{equation}

\end{proposition}

In particular we use in the sequel a property of approximation by
\textit{nondecreasing sequences}:

\begin{proposition}
\label{4bhP5} Let $\mu\in\mathcal{M}_{b}^{+}(Q)$. Let $\left\{  \mu
_{n}\right\}  $ be a nondecreasing sequence in $\mathcal{M}_{b}^{+}(Q)$
converging to $\mu$ in $\mathcal{M}_{b}(Q)$. Then, there exist $f_{n},f\in
L^{1}(Q)$, $g_{n},g\in(L^{p^{\prime}}(Q))^{N}$ and $h_{n},h\in L^{p}%
(0,T;W_{0}^{1,p}(\Omega)),$ $\mu_{n,s},\mu_{s}\in\mathcal{M}_{s}^{+}(Q)$ such
that
\[
\mu=f-\operatorname{div}g+h_{t}+\mu_{s},\qquad\mu_{n}=f_{n}-\operatorname{div}%
g_{n}+(h_{n})_{t}+\mu_{n,s},
\]
and $\left\{  f_{n}\right\}  ,\left\{  g_{n}\right\}  ,\left\{  h_{n}\right\}
$ strongly converge to $f,g,h$ in $L^{1}(Q),(L^{p^{\prime}}(Q))^{N}$ and
$L^{p}(0,T;W_{0}^{1,p}(\Omega))$ respectively, and $\left\{  \mu
_{n,s}\right\}  $ converges to $\mu_{s}$ (strongly) in $\mathcal{M}_{b}(Q)$
and
\begin{equation}
||f_{n}||_{L^{1}(Q)}+||g_{n}||_{(L^{p^{\prime}}(Q))^{N}}+||h_{n}%
||_{L^{p}(0,T;W_{0}^{1,p}(\Omega))}+\mu_{n,s}(\Omega)\leq2\mu(Q).
\label{4bh2504}%
\end{equation}

\end{proposition}

As a consequence of the above results, we get the following:

\begin{corollary}
\label{051120131} (i) Let $u_{0}\in L^{1}(\Omega)$ and $\mu\in\mathcal{M}%
_{b}(Q)$. Then there exists a R-solution $u$ to the problem \ref{pmu} with
data $(\mu,u_{0})$ such that $u$ satisfies (\ref{4bh1810131}).

(ii) Furthermore, if $v_{0}\in L^{1}(\Omega)$ and $\nu\in\mathcal{M}_{b}(Q)$
such that $u_{0}\leq v_{0}$ and $\mu\leq\nu,$ then one can find R-solutions
$u$ and $v$ to the problem \ref{pmu} with respective data $(\mu,u_{0})$ and
$(\omega,v_{0})$ such that $u\leq v$, $u$ satisfies (\ref{4bh1810131}) and
\begin{equation}
\mathrm{meas}\left\{  {|v|>k}\right\}  \leq C(||v_{0}||_{L^{1}(\Omega)}%
+|\nu|(Q))^{\frac{p+N}{N}}{k^{-p_{c}}},\qquad\forall k>0. \label{4bh1810132}%
\end{equation}

\end{corollary}

\begin{proof}
(i) We approximate $\mu$ by a smooth sequence $\left\{  \mu_{n}\right\}  $
defined at Proposition \ref{4bhatt}-(ii) and apply Proposition \ref{aid} and
Theorem \ref{sta}.

(ii) We set $w_{0}=v_{0}-u_{0}\geq0$ and $\lambda=\omega-\mu\geq0.$ In the
same way, we consider a nonnegative, smooth sequence $(\lambda_{n},w_{0,n})$
of approximations of $(\lambda,w_{0})$ defined at Proposition \ref{4bhatt}%
-(ii). Let $v_{n}$ be the solution of the problem with data $(\lambda_{n}%
+\mu_{n},w_{0,n}+u_{0,n}).$ Clearly, $u_{n}\leq v_{n}$ and $(\lambda_{n}%
+\mu_{n},w_{0,n}+u_{0,n})$ is an approximation of data $(\omega,v_{0})$ in the
sense of Theorem \ref{sta}, then we reach the conclusion.
\end{proof}

\section{Subcritical case \label{subc}}

We first consider the subcritical case with absorption. We obtain Theorem
\ref{newa} as a direct consequence of Theorem \ref{sta} and Proposition
\ref{4bhatt}. We follow the well-known technique introduced in \cite{BBGGPV}
for the elliptic problem with absorption
\begin{equation}
-\mathcal{A}(u)+G(u)=\mathcal{\omega}\quad\text{in }\Omega,\qquad
u=0\quad\text{on }\partial\Omega, \label{ella}%
\end{equation}
where $\omega\in\mathcal{M}_{b}(\Omega),p>1$, and $G$ is nondecreasing and
odd, and $\int_{1}^{\infty}G(s)s^{-(N-1)p/(N-p)}ds<\infty.$\medskip

\begin{proof}
[Proof of Theorem \ref{newa}]\ Let $\mu=\mu_{0}+\mu_{s}\in\mathcal{M}_{b}(Q)$,
with $\mu_{0}\in\mathcal{M}_{0}(Q),\mu_{s}\in\mathcal{M}_{s}(Q),$ and
$u_{0}\in L^{1}(\Omega).$ By Proposition \ref{4bhatt}, we can find
$f_{n,i},g_{n,i},h_{n,i}\in C_{c}^{\infty}(Q)$ which strongly converge to
$f_{i},g_{i},h_{i}$ in $L^{1}(Q),(L^{p^{\prime}}(Q))^{N}$ and $L^{p}%
((0,T);W_{0}^{1,p}(\Omega))$ respectively, for $i=1,2,$ such that $\mu_{0}%
^{+}=(f_{1},g_{1},h_{1}),$ $\mu_{0}^{-}=(f_{2},g_{2},h_{2}),$ and $\mu
_{n,0,i}=(f_{n,i},g_{n,i},h_{n,i})$, converge respectively for $i=1,2$ to
$\mu_{0}^{+},$ $\mu_{0}^{-}$ in the narrow topology; and we can find
nonnegative $\mu_{n,s,i}\in C_{c}^{\infty}(Q),i=1,2,$ converging respectively
to $\mu_{s}^{+},\mu_{s}^{-}$ in the narrow topology.

\noindent Furthermore, if we set
\[
\mu_{n}=\mu_{n,0,1}-\mu_{n,0,2}+\mu_{n,s,1}-\mu_{n,s,2},
\]
then $|\mu_{n}|(Q)\leq|\mu|(Q)$. Consider a sequence $\left\{  u_{0,n}%
\right\}  \subset C_{c}^{\infty}(\Omega)$ which strongly converges to $u_{0}$
in $L^{1}(\Omega)$ and satisfies $||u_{0,n}||_{1,\Omega}\leq||u_{0}%
||_{L^{1}(\Omega)}$. \medskip

\noindent Let $u_{n}$ be a solution of
\[
\left\{
\begin{array}
[c]{l}%
(u_{n})_{t}-\mathcal{A}(u_{n})+\mathcal{G}(u_{n})=\mu_{n}\qquad\text{in }Q,\\
u_{n}=0\qquad\text{on }\partial\Omega\times(0,T),\\
u_{n}(0)=u_{0,n}\qquad\text{in }\Omega.
\end{array}
\right.
\]
We can choose $\varphi=\varepsilon^{-1}T_{\varepsilon}(u_{n})$ as test
function of above problem. Since
\[
\int_{Q}{{{\left(  {{\varepsilon^{-1}\overline{T_{\varepsilon}}}({u_{n}}%
)}\right)  }_{t}}dxdt}=\int_{\Omega}{{\varepsilon^{-1}\overline{T_{\varepsilon
}}}({u_{n}}(T))dx}-\int_{\Omega}{{\varepsilon^{-1}\overline{T_{\varepsilon}}%
}({u_{0,n}})dx}\geq-||{u_{0,n}}|{|_{{L^{1}}(\Omega)}},
\]
there holds from (\ref{condi3})
\[
\int_{Q}{\mathcal{G}(x,t,{u_{n}}){\varepsilon^{-1}}{T_{\varepsilon}}({u_{n}%
})dxdt}\leq|{\mu_{n}}|(Q)+||{u_{0,n}}|{|_{{L^{1}}(\Omega)}}\leq|{\mu
}|(Q)+||{u_{0}}||_{L^{1}(\Omega)}.
\]
Letting $\varepsilon\rightarrow0$, we obtain
\[
\int_{Q}{\left\vert {\mathcal{G}(x,t,{u_{n}})}\right\vert dxdt}\leq|{\mu
}|(Q)+||{u_{0}}||_{L^{1}(\Omega)}.
\]
Next we apply the estimate (\ref{4bh1810131}) of Proposition \ref{aid} to
$u_{n},$ with initial data $u_{0,n}$ and measure data $\mu_{n}-\mathcal{G}%
(u_{n})\in L^{1}(Q)$. We get for any $s>0$ and any $n\in\mathbb{N},$
\[
\mathrm{meas}\left\{  {|u_{n}|\geq s}\right\}  \leq M{s^{-p_{c}}},\qquad
M=C(|{\mu}|(Q)+||{u_{0}}||_{{L^{1}}(\Omega)})^{\frac{p+N}{N}},\quad
C=C(N,p,\Lambda_{1},\Lambda_{2}).
\]
For any $L>1,$ we set $G_{L}(s)={\chi_{\lbrack L,\infty)}}(s)G(s),$ and
$|{u_{n}}|^{\ast}(s)=\inf\{a>0:\mathrm{meas}\left\{  |{u_{n}}|>a\right\}  \leq
s\}.$ For any $s\geq0,$ we obtain
\begin{equation}
\int\limits_{\left\{  |u_{n}|\geq L\right\}  }G(|u_{n}|){dxdt}=\int_{Q}%
{G_{L}(|u_{n}|)dxdt}\leq\int_{0}^{\infty}{{G_{L}}(|u_{n}|^{\ast}(s))ds}
\label{tuca}%
\end{equation}
Since $|\mathcal{G}(x,t,u_{n})|\leq G(|u_{n}|)$, we deduce that
$\{|\mathcal{G}(u_{n})|\}$ is equi-integrable. Then, from Proposition
\ref{4bhmun}, up to a subsequence, $\{u_{n}\}$ converges to some function $u,$
$a.e.$ in $Q,$ and $\left\{  \mathcal{G}(u_{n})\right\}  $ converges to
$\mathcal{G}(u)$ in $L^{1}(Q)$. Therefore, applying Theorem \ref{sta}, $u$ is
a R-solution of (\ref{pro0}).\bigskip
\end{proof}

Next we study the subcritical case with a source term. We proceed by induction
by constructing an nondecreasing sequence of solutions. Here we meet a
difficulty, due to the possible nonuniqueness of the solutions, that we solve
by using Corollary \ref{051120131}.$\medskip$

\begin{proof}
[Proof of Theorem \ref{news}]Let $\{u_{n}\}_{n\geq1}$ be defined by induction
as nonnegative R-solutions of
\[
\left\{
\begin{array}
[c]{l}%
(u_{1})_{t}-\mathcal{A}(u_{1})=\mu\qquad\text{in }Q,\\
u_{1}=0\qquad\text{on }\partial\Omega\times(0,T),\\
u_{1}(0)=u_{0}\qquad\text{in }\Omega,
\end{array}
\right.  \qquad\left\{
\begin{array}
[c]{l}%
(u_{n+1})_{t}-\mathcal{A}(u_{n+1})=\mu+\lambda\mathcal{G}(u_{n})\qquad\text{in
}Q,\\
u_{n+1}=0\qquad\text{on }\partial\Omega\times(0,T),\\
u_{n+1}(0)=u_{0}\qquad\text{in }\Omega,
\end{array}
\right.
\]
From Corollary \ref{051120131} we can assume that $\{u_{n}\}$ is nondecreasing
and satisfies, for any $s>0$ and $n\in\mathbb{N}$
\begin{equation}
\mathrm{meas}\left\{  {|u_{n}|\geq s}\right\}  \leq C_{1}K_{n}{s^{-p_{c}}},
\label{mnn}%
\end{equation}
where $C_{1}$ does not depend on $s,n,$ and
\begin{align*}
&  K_{1}=(||u_{0}||_{L^{1}(\Omega)}+|\mu|(Q))^{\frac{p+N}{N}},\\
K_{n+1}  &  =(||u_{0}||_{L^{1}(\Omega)}+|\mu|(Q)+\lambda||\mathcal{G}%
(u_{n})||_{L^{1}(\Omega)})^{\frac{p+N}{N}},
\end{align*}
for any $n\geq1.$ Take $\varepsilon=\lambda+|\mu|(Q)+||u_{0}||_{L^{1}(\Omega
)}\leq1$. Denoting by $C_{i}$ some constants independent on $n,\varepsilon,$
there holds $K_{1}\leq C_{2}\varepsilon,$ and for $n\geq1,$
\[
K{_{n+1}}\leq C_{3}\varepsilon(||\mathcal{G}({u_{n}})||_{L^{1}(\Omega
)}^{1+\frac{p}{N}}+{1)}.
\]
From (\ref{tuca}) and (\ref{mnn}), we find
\[
{\left\Vert {\mathcal{G}({u_{n}})}\right\Vert _{{L^{1}}({Q})}}\leq\left\vert
{{Q}}\right\vert G(2)+\int\limits_{\left\{  {u_{n}}\geq2\right\}  |}{G({u_{n}%
})dxdt}\leq\left\vert {{Q}}\right\vert G(2)+C_{4}K{_{n}\int_{2}^{\infty
}G\left(  s\right)  {s^{-1-p_{c}}}ds.}%
\]
Thus, ${K_{n+1}}\leq{C}_{5}\varepsilon(K_{n}^{1+\frac{p}{N}}+{1)}$. Therefore,
if $\varepsilon$ is small enough, $\left\{  K_{n}\right\}  $ is bounded. Since
$\{u_{n}\}$ is nondecreasing, from (\ref{tuca}) and the relation
$\mathcal{G}(x,t,u_{n})\leq G(u_{n})$, we deduce that $\{\mathcal{G}(u_{n})\}$
converges. Then by Theorem \ref{sta}, up to a subsequence, $\{u_{n}\}$
converges to a R-solution $u$ of (\ref{pro1}).\medskip
\end{proof}

\begin{remark}
Theorems \ref{newa} and \ref{news} are still valid for operators $\mathcal{A}$
also depending on $t$, satisfying conditions analogous to (\ref{condi3}),
(\ref{condi4}).
\end{remark}

\section{General case with absorption terms}

In the sequel we combine the results of Theorem \ref{sta} with delicate
techniques introduced in \cite{BiNQVe} for the elliptic problem (\ref{ella}),
for proving Theorems \ref{main1} and \ref{expo}. In these proofs the use of
the elliptic Wolff potential is an essential tool.\medskip

We recall a first result obtained in \cite[Corollary 3.4 and Theorem
3.8]{BiNQVe} for the elliptic problem without perturbation term, inspired from
\cite[Theorem 2.1]{PhVe1}:

\begin{theorem}
\label{elli} Let $1<p<N$, $\Omega$ be a bounded domain of $\mathbb{R}^{N}$ and
$\omega\in\mathcal{M}_{b}(\Omega)$ with compact support in $\Omega$. Suppose
that $u_{n}$ is a solution of problem%
\[
\left\{
\begin{array}
[c]{l}%
-\mathcal{A}(u_{n})=\varphi_{n}\ast\omega\qquad\text{in }\Omega,\\
{u}_{n}=0\qquad\text{on }\partial\Omega,
\end{array}
\right.
\]
where $\{\varphi_{n}\}$ is a sequence of mollifiers in $\mathbb{R}^{N}$. Then,
up to subsequence, $u_{n}$ converges a.e in $\Omega$ to a renormalized
solution $u$ of%
\[
\left\{
\begin{array}
[c]{l}%
-\mathcal{A}(u)=\omega\qquad\text{in }\Omega,\\
{u}=0\qquad\text{on }\partial\Omega,
\end{array}
\right.
\]
in the elliptic sense of \cite{DMOP}, satisfying
\begin{equation}
-\kappa\mathbf{W}_{1,p}^{2D}[\omega^{-}]\leq u\leq\kappa\mathbf{W}_{1,p}%
^{2D}[\omega^{+}] \label{enca}%
\end{equation}
where $\kappa$ is a constant which only depends of $N,p,\Lambda_{1}%
,\Lambda_{2}.$
\end{theorem}

Next we give a general result for the parabolic problem (\ref{abp}) with absorption:

\begin{theorem}
\label{main} Let $p<N$, and assume that $s\mapsto\mathcal{G}(x,t,s)$ is
nondecreasing and odd, for $a.e.$ $(x,t)$ in $Q$.

\noindent Let $\mu_{1},\mu_{2}\in\mathcal{M}_{b}^{+}(Q)$ such that there exist
$\left\{  \omega_{n}\right\}  \subset\mathcal{M}_{b}^{+}(\Omega)$ and
nondecreasing sequences $\left\{  \mu_{1,n}\right\}  ,\left\{  \mu
_{2,n}\right\}  $ in $\mathcal{M}_{b}^{+}(Q)$ with compact support in $Q$,
converging to $\mu_{1},\mu_{2}$, respectively in the narrow topology, and
satisfying
\[
\mu_{1,n},\mu_{2,n}\leq\omega_{n}\otimes\chi_{(0,T)},\quad\text{and\quad
}\mathcal{G}((n+\kappa\mathbf{W}_{1,p}^{2D}\left[  {\omega_{n}}\right]  ))\in
L^{1}(Q),
\]
where the constant $\kappa$ is given at Theorem \ref{elli}. Let $u_{0}\in
L^{1}(\Omega)$, and $\mu=$ $\mu_{1}-\mu_{2}.\medskip$

Then there exists a R-solution $u$ of problem (\ref{abp}). Moreover if
$u_{0}\in L^{\infty}(\Omega),$ and $\omega_{n}\leq\gamma$ for any
$n\in\mathbb{N}$, for some $\gamma\in\mathcal{M}_{b}^{+}(\Omega),$ then $a.e.$
in $Q$,
\begin{equation}
\left\vert u(x,t)\right\vert \leq\kappa\mathbf{W}_{1,p}^{2D}\left[
\gamma\right]  (x)+||u_{0}||_{L^{\infty}(\Omega)}. \label{abs}%
\end{equation}

\end{theorem}

For proving this result, we need two Lemmas:

\begin{lemma}
\label{T26} Let $\mathcal{G}$ satisfy the assumptions of Theorem \ref{main}
and $\mathcal{G}\in L^{\infty}(Q\times\mathbb{R})$. For $i=1,2,$ let
$u_{0,i}\in L^{\infty}(\Omega)$ be nonnegative, and $\lambda_{i}=\lambda
_{i,0}+\lambda_{i,s}\in\mathcal{M}_{b}^{+}(Q)$ with compact support in $Q$,
$\gamma\in\mathcal{M}_{b}^{+}(\Omega)$ with compact support in $\Omega$ such
that $\lambda_{i}\leq\gamma\otimes\chi_{(0,T)}$. Let $\lambda_{i,0}%
=(f_{i},g_{i},h_{i})$ be a decomposition of $\lambda_{i,0}$ into functions
with compact support in $Q$.$\medskip$

Then, there exist R-solutions $u,u_{1},u_{2},$ to problems
\begin{equation}
\left\{
\begin{array}
[c]{l}%
u_{t}-\mathcal{A}(u)+\mathcal{G}(u)=\lambda_{1}-\lambda_{2}\qquad\text{in
}Q,\\
u=0\qquad\text{on }\partial\Omega\times(0,T),\\
u(0)=u_{0,1}-u_{0,2},\qquad\text{in }\Omega,
\end{array}
\right.  \label{proba}%
\end{equation}%
\begin{equation}
\left\{
\begin{array}
[c]{l}%
(u_{i})_{t}-\mathcal{A}(u_{i})+\mathcal{G}(u_{i})=\lambda_{i}\qquad\text{in
}Q,\\
u_{i}=0\qquad\text{on }\partial\Omega\times(0,T),\\
u_{i}(0)=u_{0,i},\qquad\text{in }\Omega,
\end{array}
\right.  \quad\label{probb}%
\end{equation}
relative to decompositions $(f_{1,n}-f_{2,n}-\mathcal{G}(u_{n}),g_{1,n}%
-g_{2,n},h_{1,n}-h_{2,n})$, $(f_{i,n}-\mathcal{G}(u_{i,n}),g_{i,n},h_{i,n}),$
such that $a.e.$ in $Q,$
\begin{equation}
-||u_{0,2}||_{L^{\infty}(\Omega)}-\kappa\mathbf{W}_{1,p}^{2D}\left[
\gamma\right]  (x)\leq-u_{2}(x,t)\leq u(x,t)\leq u_{1}(x,t)\leq\kappa
\mathbf{W}_{1,p}^{2D}\left[  \gamma\right]  (x)+||u_{0,1}||_{L^{\infty}%
(\Omega)}, \label{17062}%
\end{equation}
and
\begin{equation}
\int_{Q}\left\vert \mathcal{G}(u)\right\vert {dxdt}\leq\sum_{i=1,2}\left(
\lambda_{i}(Q)+||u_{0,i}||_{L^{1}(\Omega)}\right)  \quad\text{and \quad}%
\int_{Q}\mathcal{G}(u_{i}){dxdt}\leq\lambda_{i}(Q)+||u_{0,i}||_{L^{1}(\Omega
)},\quad i=1,2. \label{18061}%
\end{equation}
\medskip Furthermore, assume that $\mathcal{H},\mathcal{K}$ have the same
properties as $\mathcal{G},$ and $\mathcal{H}(x,t,s)\leq\mathcal{G}%
(x,t,s)\leq\mathcal{K}(x,t,s)$ for any $s\in(0,+\infty)$ and $a.e.$ in $Q.$
Then, one can find solutions $u_{i}(\mathcal{H}),u_{i}(\mathcal{K})$,
corresponding to $\mathcal{H},\mathcal{K}$ with data $\lambda_{i}$, such that
$u_{i}(\mathcal{H})\geq u_{i}\geq u_{i}(\mathcal{K})$, $i=1,2$.\medskip

\noindent Assume that $\omega_{i},\theta_{i}$ have the same properties as
$\lambda_{i}$ and $\omega_{i}\leq\lambda_{i}\leq\theta_{i}$, $u_{0,i,1}%
,u_{0,i,2}\in L^{\infty+}(\Omega)$, $u_{0,i,2}\leq u_{0,i}\leq u_{0,i,1}.$
Then one can find solutions $u_{i}(\omega_{i}),u_{i}(\theta_{i})$,
corresponding to $(\omega_{i},u_{0,i,2}),(\theta_{i},u_{0,i,1})$, such that
$u_{i}(\omega_{i},u_{0,i,2})\leq u_{i}\leq u_{i}(\theta_{i},u_{0,i,1}%
).\medskip$
\end{lemma}

\begin{proof}
Let $\left\{  \varphi_{1,n}\right\}  ,\left\{  \varphi_{2,n}\right\}  $ be
sequences of mollifiers in $\mathbb{R}$ and $\mathbb{R}^{N}$, and $\varphi
_{n}=\varphi_{1,n}\varphi_{2,n}$. Set $\gamma_{n}=\varphi_{2,n}\ast\gamma,$
and for $i=1,2,$ $u_{0,i,n}=\varphi_{2,n}\ast u_{0,i},$
\[
\lambda_{i,n}=\varphi_{n}\ast\lambda_{i}=f_{i,n}-\text{div}(g_{i,n}%
)+(h_{i,n})_{t}+\lambda_{i,s,n},
\]
where $f_{i,n}=\varphi_{n}\ast f_{i},\;g_{i,n}=\varphi_{n}\ast g_{i}%
,\;h_{i,n}=\varphi_{n}\ast h_{i},\;\lambda_{i,s,n}=\varphi_{n}\ast
\lambda_{i,s},$ and
\[
\lambda_{n}=\lambda_{1,n}-\lambda_{2,n}=f_{n}-\text{div}(g_{n})+(h_{n}%
)_{t}+\lambda_{s,n},
\]
where $f_{n}=f_{1,n}-f_{2,n},\;g_{n}=g_{1,n}-g_{2,n},\;h_{n}=h_{1,n}%
-h_{2,n},\;\lambda_{s,n}=\lambda_{1,s,n}-\lambda_{2,s,n}$. Then for $n$ large
enough, $\lambda_{1,n},\lambda_{2,n},\lambda_{n}\in C_{c}^{\infty}(Q)$,
$\gamma_{n}\in C_{c}^{\infty}(\Omega).$ Thus there exist unique solutions
$u_{n},u_{i,n},v_{i,n},$ $i=1,2,$ of problems
\[
\left\{
\begin{array}
[c]{l}%
(u_{n})_{t}-\mathcal{A}(u_{n})+\mathcal{G}(u_{n})=\lambda_{1,n}-\lambda
_{2,n}\qquad\text{in }Q,\\
u_{n}=0\qquad\text{on }\partial\Omega\times(0,T),\\
u_{n}(0)=u_{0,1,n}-u_{0,2,n}\qquad\text{in }\Omega,
\end{array}
\right.
\]%
\[
\left\{
\begin{array}
[c]{l}%
(u_{i,n})_{t}-\mathcal{A}(u_{i,n})+\mathcal{G}(u_{i,n})=\lambda_{i,n}%
\qquad\text{in }Q,\\
u_{i,n}=0\qquad\text{on }\partial\Omega\times(0,T),\\
u_{i,n}(0)=u_{0,i,n}\qquad\text{in }\Omega,
\end{array}
\right.
\]%
\[
-\mathcal{A}(w_{n})=\gamma_{n}\quad\text{in }\Omega,\qquad w_{n}%
=0\quad\text{on }\partial\Omega,
\]
such that
\[
-||u_{0,2}||_{L^{\infty}(\Omega)}-w_{n}(x)\leq-u_{2,n}(x,t)\leq u_{n}(x,t)\leq
u_{1,n}(x,t)\leq w_{n}(x)+||u_{0,1}||_{L^{\infty}(\Omega)},\quad a.e.\text{ in
}Q.
\]
Otherwise, as in the Proof of Theorem \ref{newa}, (i), there holds
\[
\int_{Q}|\mathcal{G}(u_{n})|{dxdt}\;\leq\sum_{i=1,2}\left(  \lambda
_{i}(Q)+||u_{0,i,n}||_{L^{1}(\Omega)}\right)  ,\quad\text{and\quad}\int%
_{Q}\mathcal{G}(u_{i,n}){dxdt}\leq\lambda_{i}(Q)+||u_{0,i,n}||_{L^{1}(\Omega
)},\quad i=1,2.
\]
From Proposition \ref{4bhmun}, up to a common subsequence, $\left\{
u_{n},u_{1,n},u_{2,n}\right\}  $ converge to some $(u,u_{1},u_{2}),$ $a.e.$ in
$Q$. Since $\mathcal{G}$ is bounded, in particular, $\left\{  \mathcal{G}%
(u_{n})\right\}  $ converges to $\mathcal{G}(u)$ and $\left\{  \mathcal{G}%
(u_{i,n})\right\}  $ converges to $\mathcal{G}(u_{i})$ in $L^{1}(Q)$. Thus,
(\ref{18061}) is satisfied. Moreover $\left\{  \lambda_{i,n}-\mathcal{G}%
(u_{i,n}),f_{i,n}-\mathcal{G}(u_{i,n}),g_{i,n},h_{i,n},\lambda_{i,s,n}%
,u_{0,i,n}\right\}  $ is an approximation of $(\lambda_{i}-\mathcal{G}%
(u_{i}),f_{i}-\mathcal{G}(u_{i}),g_{i},h_{i},\lambda_{i,s},u_{0,i}),$ and
$\left\{  \lambda_{n}-\mathcal{G}(u_{n}),f_{n}-\mathcal{G}(u_{n}),g_{n}%
,h_{n},\lambda_{s,n},u_{0,1,n}-u_{0,2,n}\right\}  $ is an approximation of
$(\lambda_{1}-\lambda_{2}-\mathcal{G}(u),f-\mathcal{G}(u),g,h,\lambda
_{s},u_{0,1}-u_{0,2})$, in the sense of Theorem \ref{sta}. Thus, we can find
(different) subsequences converging $a.e.$ to $u,u_{1},u_{2},$ R-solutions of
(\ref{proba}) and (\ref{probb}). Furthermore, from Theorem \ref{elli}, up to a
subsequence, $\left\{  w_{n}\right\}  $ converges $a.e.$ in $Q$ to a
renormalized solution of
\[
-\mathcal{A}(w)=\gamma\quad\text{in }\Omega,\qquad w=0\quad\text{on }%
\partial\Omega,
\]
such that $w\leq\kappa\mathbf{W}_{1,p}^{2D}\left[  {\gamma}\right]  {,}$
$a.e.$ in $\Omega$. Hence, we get the inequality (\ref{17062}). The other
conclusions follow in the same way.\medskip
\end{proof}

\begin{lemma}
\label{T27} Let $\mathcal{G}$ satisfy the assumptions of Theorem \ref{main}.
For $i=1,2,$ let $u_{0,i}\in L^{\infty}(\Omega)$ be nonnegative, $\lambda
_{i}\in\mathcal{M}_{b}^{+}(Q)$ with compact support in $Q$, and $\gamma
\in\mathcal{M}_{b}^{+}(\Omega)$ with compact support in $\Omega$, such that
\begin{equation}
\lambda_{i}\leq\gamma\otimes\chi_{(0,T)},\quad\text{and\quad}\mathcal{G}%
((||u_{0,i}||_{L^{\infty}(\Omega)}+\kappa\mathbf{W}_{1,p}^{2D}\left[
\gamma\right]  ))\in L^{1}(Q). \label{hypg}%
\end{equation}
Let $\lambda_{i,0}=(f_{i},g_{i},h_{i})$ be a decomposition of $\lambda_{i,0}$
into functions with compact support in $Q.\medskip$

Then, there exist R-solutions $u,u_{1},u_{2}$ of the problems (\ref{proba})
and (\ref{probb}), respectively relative to the decompositions $(f_{1}%
-f_{2}-\mathcal{G}(u),g_{1}-g_{2},h_{1}-h_{2})$, $(f_{i}-\mathcal{G}%
(u_{i}),g_{i},h_{i}),$ satifying (\ref{17062}) and (\ref{18061}).\medskip

\quad Moreover, assume that $\omega_{i},\theta_{i}$ have the same properties
as $\lambda_{i}$ and $\omega_{i}\leq\lambda_{i}\leq\theta_{i}$, $u_{0,i,1}%
,u_{0,i,2}\in L^{\infty}(\Omega)$, $0\leq u_{0,i,2}\leq u_{0,i}\leq
u_{0,i,1}.$ Then, one can find solutions $u_{i}(\omega_{i},u_{0,i,2}),$
$u_{i}(\theta_{i},u_{0,i,1})$, corresponding with $(\omega_{i},u_{0,i,2}),$
$(\theta_{i},u_{0,i,1})$, such that $u_{i}(\omega_{i},u_{0,i,2})\leq u_{i}\leq
u_{i}(\theta_{i},u_{0,i,1}).\medskip$
\end{lemma}

\begin{proof}
From Lemma \ref{T26} there exist R-solutions $u_{n}$, $u_{i,n}$ to problems
\[
\left\{
\begin{array}
[c]{l}%
(u_{n})_{t}-\mathcal{A}(u_{n})+T_{n}(\mathcal{G}(u_{n}))=\lambda_{1}%
-\lambda_{2}\qquad\text{in }Q,\\
u_{n}=0\qquad\text{on }\partial\Omega\times(0,T),\\
u_{n}(0)=u_{0,1}-u_{0,2}\qquad\text{in }\Omega,
\end{array}
\right.
\]%
\[
\left\{
\begin{array}
[c]{l}%
(u_{i,n})_{t}-\mathcal{A}(u_{i,n})+T_{n}(\mathcal{G}(u_{i,n}))=\lambda
_{i}\qquad\text{in }Q,\\
u_{i,n}=0\qquad\text{on }\partial\Omega\times(0,T),\\
u_{i,n}(0)=u_{0,i},\qquad\text{in }\Omega,
\end{array}
\right.
\]
relative to the decompositions $(f_{1}-f_{2}-T_{n}(\mathcal{G}(u_{n}%
)),g_{1}-g_{2},h_{1}-h_{2})$, $(f_{i}-T_{n}(\mathcal{G}(u_{i,n})),g_{i}%
,h_{i});$ and they satisfy, $a.e.$ in $Q,$
\begin{equation}
-||u_{0,2}||_{L^{\infty}(\Omega)}-\kappa\mathbf{W}_{1,p}^{2D}\left[
\gamma\right]  (x)\leq-u_{2,n}(x,t)\leq u_{n}(x,t)\leq u_{1,n}(x,t)\leq
\kappa\mathbf{W}_{1,p}^{2D}{\gamma}(x)+||u_{0,1}||_{L^{\infty}(\Omega)},
\label{all}%
\end{equation}%
\[
\int_{Q}|T_{n}\left(  \mathcal{G}(u_{n})\right)  |{dxdt}\leq\sum
_{i=1,2}(\lambda_{i}(Q)+||u_{0,i}||_{L^{1}(\Omega)}),\quad\text{and\quad}%
\int_{Q}T_{n}\left(  \mathcal{G}(u_{i,n})\right)  {dxdt}\leq\lambda
_{i}(Q)+||u_{0,i}||_{L^{1}(\Omega)}.
\]
As in Lemma \ref{T26}, up to a common subsequence, $\{u_{n},u_{1,n},u_{2,n}\}$
converges $a.e.$ in $Q$ to $\{u,u_{1},u_{2}\}$ for which (\ref{17062}) is
satisfied $a.e.$ in $Q$. From (\ref{hypg}), (\ref{all}) and the dominated
convergence Theorem, we deduce that $\left\{  T_{n}(\mathcal{G}(u_{n}%
))\right\}  $ converges to $\mathcal{G}(u)$ and $\left\{  T_{n}(\mathcal{G}%
(u_{i,n}))\right\}  $ converges to $\mathcal{G}(u_{i})$ in $L^{1}(Q)$. Thus,
from Theorem \ref{sta}, $u$ and $u_{i}$ are respective R-solutions of
(\ref{proba}) and (\ref{probb}) relative to the decompositions $(f_{1}%
-f_{2}-\mathcal{G}(u),g_{1}-g_{2},h_{1}-h_{2})$, $(f_{i}-\mathcal{G}%
(u_{i}),g_{i},h_{i}),$ and (\ref{17062}) and (\ref{18061}) hold. The last
statement follows from the same assertion in Lemma \ref{T26}.\medskip
\end{proof}

\begin{proof}
[Proof of Theorem \ref{main}]By Proposition \ref{4bhP5}, for $i=1,2,$ there
exist $f_{i,n},f_{i}\in L^{1}(Q)$, $g_{i,n},g_{i}\in(L^{p^{\prime}}(Q))^{N}$
and $h_{i,n},h_{i}\in L^{p}((0,T);W_{0}^{1,p}(\Omega)),$ $\mu_{i,n,s}%
,\mu_{i,s}\in\mathcal{M}_{s}^{+}(Q)$ such that
\[
\mu_{i}=f_{i}-\operatorname{div}g_{i}+(h_{i})_{t}+\mu_{i,s},\qquad\mu
_{i,n}=f_{i,n}-\operatorname{div}g_{i,n}+(h_{i,n})_{t}+\mu_{i,n,s},
\]
and $\left\{  f_{i,n}\right\}  ,\left\{  g_{i,n}\right\}  ,\left\{
h_{i,n}\right\}  $ strongly converge to $f_{i},g_{i},h_{i}$ in $L^{1}%
(Q),\;(L^{p^{\prime}}(Q))^{N}$ and $L^{p}((0,T);W_{0}^{1,p}(\Omega))$
respectively, and $\left\{  \mu_{i,n}\right\}  ,\left\{  \mu_{i,n,s}\right\}
$ converge to $\mu_{i},\mu_{i,s}$ (strongly) in $\mathcal{M}_{b}(Q),$ and
\[
||f_{i,n}||_{L^{1}(\Omega)}+||g_{i,n}||_{L^{p^{\prime}}(\Omega)}%
+||h_{i,n}||_{L^{p}((0,T);W_{0}^{1,p}(\Omega))}+\mu_{i,n,s}(\Omega)\leq
2\mu(Q).
\]
By Lemma \ref{T27}, there exist R-solutions $u_{n}$, $u_{i,n}$ to problems
\[
\left\{
\begin{array}
[c]{l}%
(u_{n})_{t}-\mathcal{A}(u_{n})+\mathcal{G}(u_{n})=\mu_{1,n}-\mu_{2,n}%
\qquad\text{in }Q,\\
u_{n}=0\qquad\text{on }\partial\Omega\times(0,T),\\
u_{n}(0)=T_{n}(u_{0})\qquad\text{in }\Omega,
\end{array}
\right.
\]%
\[
\left\{
\begin{array}
[c]{l}%
(u_{i,n})_{t}-\mathcal{A}(u_{i,n})+\mathcal{G}(u_{i,n})=\mu_{i,n}%
\qquad\text{in }Q,\\
u_{i,n}=0\qquad\text{on }\partial\Omega\times(0,T),\\
u_{i,n}(0)=T_{n}(u_{0}^{\pm})\qquad\text{in }\Omega,
\end{array}
\right.
\]
for $i=1,2,$ relative to the decompositions $(f_{1,n}-f_{2,n}-\mathcal{G}%
(u_{n}),g_{1,n}-g_{2,n},h_{1,n}-h_{2,n})$, $(f_{i,n}-\mathcal{G}%
(u_{i,n}),g_{i,n},h_{i,n}),$ such that $\{u_{i,n}\}$ is nonnegative and
nondecreasing, and $-u_{2,n}\leq u_{n}\leq u_{1,n}$; and
\begin{equation}
\int_{Q}|\mathcal{G}(u_{n})|dxdt,\int_{Q}\mathcal{G}(u_{i,n})dxdt\leq\mu
_{1}(Q)+\mu_{2}(Q)+||u_{0}||_{L^{1}(\Omega)}. \label{aaa}%
\end{equation}
\newline As in the proof of Lemma \ref{T27}, up to a common subsequence
$\{u_{n},u_{1,n},u_{2,n}\}$ converge $a.e.$ in $Q$ to $\{u,u_{1},u_{2}\}$.
Since $\left\{  \mathcal{G}(u_{i,n})\right\}  $ is nondecreasing, and
nonnegative, from the monotone convergence Theorem and (\ref{aaa}), we obtain
that $\left\{  \mathcal{G}(u_{i,n})\right\}  $ converges to $\mathcal{G}%
(u_{i})$ in $L^{1}(Q)$, $i=1,2$. Finally, $\left\{  \mathcal{G}(u_{n}%
)\right\}  $ converges to $\mathcal{G}(u)$ in $L^{1}(Q),$ since $|\mathcal{G}%
(u_{n})|\leq\mathcal{G}(u_{1,n})+\mathcal{G}(u_{2,n})$. Thus, we can see that
\[
\left\{  \mu_{1,n}-\mu_{2,n}-\mathcal{G}(u_{n}),f_{1,n}-f_{2,n}-\mathcal{G}%
(u_{n}),g_{1,n}-g_{2,n},h_{1,n}-h_{2,n},\mu_{1,s,n}-\mu_{2,s,n},T_{n}%
(u_{0})\right\}
\]
is an approximation of $(\mu_{1}-\mu_{2}-\mathcal{G}(u),f_{1}-f_{2}%
-\mathcal{G}(u),g_{1}-g_{2},h_{1}-h_{2},\mu_{1,s}-\mu_{2,s},u_{0})$, in the
sense of Theorem \ref{sta}. Therefore, $u$ is a R-solution of (\ref{pga}), and
(\ref{abs}) holds if $u_{0}\in L^{\infty}(\Omega)$ and $\omega_{n}\leq\gamma$
for any $n\in\mathbb{N}$ and some $\gamma\in\mathcal{M}_{b}^{+}(\Omega
).\medskip$
\end{proof}

As a consequence of Theorem \ref{main}, we get a result for problem
(\ref{pmu}), used in Section \ref{sour}:

\begin{corollary}
\label{TH5} Let $u_{0}\in L^{\infty}(\Omega),$ and $\mu\in\mathcal{M}_{b}%
(${$Q$}$)$ such that $\left\vert \mu\right\vert \leq\omega\otimes\chi_{(0,T)}$
for some $\omega\in\mathcal{M}_{b}^{+}(\Omega)$. Then there exist a R-solution
u of (\ref{pmu}), such that
\begin{equation}
\left\vert u(x,t)\right\vert \leq\kappa\mathbf{W}_{1,p}^{2D}[\omega
](x)+||u_{0}||_{L^{\infty}(\Omega)},\qquad\text{for }a.e.~(x,t)\in Q,
\label{09041}%
\end{equation}
where $\kappa$ is defined at Theorem \ref{elli}.$\medskip$
\end{corollary}

\begin{proof}
Let $\left\{  \phi_{n}\right\}  $ be a nonnegative, nondecreasing sequence in
$C_{c}^{\infty}(Q)$ which converges to $1,$ $a.e.$ in $Q.$ Since $\{\phi
_{n}\mu^{+}\},\{\phi_{n}\mu^{-}\}$ are nondecreasing sequences, the result
follows from Theorem \ref{main}.
\end{proof}

\subsection{The power case}

First recall some results relative to the elliptic case for the model problem%
\begin{equation}
-\Delta_{p}u+\left\vert u\right\vert ^{q-1}u=\mathcal{\omega}\quad\text{in
}\Omega,\qquad u=0\quad\text{on }\partial\Omega, \label{mod}%
\end{equation}
with $\omega\in\mathcal{M}_{b}(\Omega),q>p-1>0$.

For $p=2,$ it is shown in \cite{BaPi1} that (\ref{mod}) admits a solution if
and only if $\omega$ does not charge the sets of Bessel $\mathrm{Cap}%
_{\mathbf{G}_{2},\frac{q}{q-1}}$-capacity zero. For $p\neq2,$ existence holds
for any measure $\omega\in\mathcal{M}_{b}(\Omega)$ in the subcritical case%
\begin{equation}
q<p_{e}:=N(p-1)/(N-p) \label{hol}%
\end{equation}
from \cite{BBGGPV}. Some necessary conditions for existence have been given in
\cite{Bi1,Bi2}. From \cite[Theorem 1.1]{BiNQVe}, a sufficient condition for
existence is that $\omega$ does not charge the sets of $\mathrm{Cap}%
_{\mathbf{G}_{p},\frac{q}{q+1-p}}$--capacity zero, and it can be conjectured
that this condition is also necessary.\medskip

Next we prove Theorem \ref{main1}. We use the following result of
\cite{BiNQVe}:

\begin{proposition}
\label{110413} Let $q>p-1$ and $\nu\in\mathcal{M}_{b}^{+}(\Omega)$.\medskip

If $\nu$ does not charge the sets of $\mathrm{Cap}_{\mathbf{G}_{p,\frac
{q}{q+1-p}}}$-capacity zero, there exists a nondecreasing sequence $\{\nu
_{n}\}\subset\mathcal{M}_{b}^{+}(\Omega)$ with compact support in $\Omega$
which converges to $\nu$ strongly in $\mathcal{M}_{b}(\Omega)$ and such that
$\mathbf{W}_{1,p}^{R}[\nu_{n}]\in L^{q}(\mathbb{R}^{N})$, for any
$n\in\mathbb{N}$ and $R>0.$\medskip
\end{proposition}

\begin{proof}
[Proof of Theorem \ref{main1}]Let $f\in L^{1}(Q)$, $u_{0}\in L^{1}(\Omega),$
and $\mu\in\mathcal{M}_{b}(Q)$ such that $\left\vert \mu\right\vert \leq
\omega\otimes F,$ where $F\in L^{1}((0,T))$ and $\omega$ does not charge the
sets of $\mathrm{Cap}_{\mathbf{G}_{p,\frac{q}{q+1-p}}}$-capacity zero. From
Proposition \ref{110413}, there exists a nondecreasing sequence $\{\omega
_{n}\}\subset\mathcal{M}_{b}^{+}(\Omega)$ with compact support in $\Omega$
which converges to $\omega,$ strongly in $\mathcal{M}_{b}(\Omega),$ such that
$\mathbf{W}_{1,p}^{2D}[\omega_{n}]\in L^{q}(\mathbb{R}^{N})$. We can write
\begin{equation}
f+\mu=\mu_{1}-\mu_{2},\qquad\mu_{1}=f^{+}+\mu^{+},\qquad\mu_{2}=f^{-}+\mu^{-},
\label{rs}%
\end{equation}
and $\mu^{+},\mu^{-}\leq\omega\otimes F.$ We set
\begin{equation}
Q_{n}=\{(x,t)\in\Omega\times(\frac{1}{n},T-\frac{1}{n}):d(x,\partial
\Omega)>\frac{1}{n}\},\qquad F_{n}=T_{n}(\chi_{(\frac{1}{n}T-\frac{1}{n})}F),
\label{qn}%
\end{equation}%
\begin{equation}
\mu_{1,n}=T_{n}(\chi_{Q_{n}}f^{+})+\inf\{\mu^{+},\omega_{n}\otimes
F_{n}\},\qquad\mu_{2,n}=T_{n}(\chi_{Q_{n}}f^{-})+\inf\{\mu^{-},\omega
_{n}\otimes F_{n}\}. \label{sn}%
\end{equation}
Then $\left\{  \mu_{1,n}\right\}  ,\left\{  \mu_{2,n}\right\}  $ are
nondecreasing sequences with compact support in $Q,$ and
\[
\mu_{1,n},\mu_{2,n}\leq\tilde{\omega}_{n}\otimes\chi_{(0,T)},\qquad\text{with
}\tilde{\omega}_{n}=n(\chi_{\Omega}+\omega_{n}),
\]
and $(n+\kappa\mathbf{W}_{1,p}^{2D}[\tilde{\omega}_{n}])^{q}\in L^{1}(Q)$.
Besides, $\omega_{n}\otimes F_{n}$ converges to $\omega\otimes F$ strongly in
$\mathcal{M}_{b}(Q).$ Indeed we easily check that
\[
||\omega_{n}\otimes F_{n}-\omega\otimes F||_{\mathcal{M}_{b}(Q)}\leq
||F_{n}||_{L^{1}((0,T))}||\omega_{n}-\omega||_{\mathcal{M}_{b}(\Omega
)}+||\omega||_{\mathcal{M}_{b}(\Omega)}||F_{n}-F||_{L^{1}((0,T))}%
\]
Observe that for any measures $\nu,\theta,\eta\in\mathcal{M}_{b}(Q),$ there
holds
\[
\left\vert \inf\{\nu,\theta\}-\inf\{\nu,\eta\}\right\vert \leq\left\vert
\theta-\eta\right\vert ,
\]
hence $\left\{  \mu_{1,n}\right\}  ,\left\{  \mu_{2,n}\right\}  $ converge to
$\mu_{1},\mu_{2}$ respectively in $\mathcal{M}_{b}(Q)$. Therefore, the result
follows from Theorem \ref{main}.\medskip
\end{proof}

\begin{remark}
\label{mac} From Theorem \ref{main1}, we deduce the existence for any measure
$\omega\in\mathcal{M}_{b}(\Omega)$ for $p<p_{e},$ whre $p_{e}$ is defined at
(\ref{hol}), since $p_{e}$ is the critical exponent of the elliptic problem
(\ref{mod}). Note that $p_{e}>p_{c}$ since $p>p_{1}.$ Let $\mathcal{M}%
_{0,e}(\Omega)$ be the set of Radon measures $\omega$ on that do not charge
the sets of zero $c_{p}^{\Omega}$-capacity, where, for any compact set
$K\subset\Omega,$
\[
c_{p}^{\Omega}(K)=\inf\{\int_{\Omega}|\nabla\varphi|^{p}dx:\varphi\geq\chi
_{K},\varphi\in C_{c}^{\infty}(\Omega)\}.
\]
From \cite[Theorem 2.16]{DrPoPr}, for any $F\in L^{1}((0,T))$ with $\int%
_{0}^{T}F(t)dt\not =0,$ and $\omega\in\mathcal{M}_{b}(\Omega)$,
\[
\omega\in\mathcal{M}_{0,e}(\Omega)\Longleftrightarrow\omega\otimes
F\in\mathcal{M}_{0}(Q).
\]
If $q\geq p_{e},$ there exist measures $\omega\in\mathcal{M}_{b}^{+}(\Omega)$
which do not charge the sets of $\mathrm{Cap}_{\mathbf{G}_{p,\frac{q}{q+1-p}}%
}$-capacity zero, such that $\omega\not \in \mathcal{M}_{0,e}(\Omega).$ As a
consequence, Theorem \ref{main1} shows the existence for some measures
$\mu\not \in \mathcal{M}_{0}(Q).$\medskip
\end{remark}

\begin{remark}
Let $\mathcal{G}:Q\times\mathbb{R}\rightarrow\mathbb{R}$ be a Caratheodory
function such that the map $s\mapsto\mathcal{G}(x,t,s)$ is nondecreasing and
odd, for $a.e.$ $(x,t)$ in $Q$. Let $\mu\in\mathcal{M}_{b}(Q),f\in
L^{1}(Q),u_{0}\in L^{1}(\Omega)$ and $\omega\in\mathcal{M}_{b}^{+}(\Omega)$
such that (\ref{hypmu}) holds.

If $\omega(\{x:\mathbf{W}_{1,p}^{2D}[\omega](x)=\infty\})=0,$ then,
(\ref{abp}) has a R-solution with data $(f+\mu,u_{0})$. The proof is similar
to the one of Theorem \ref{main1}, after replacing $\omega_{n}$ by
$\chi_{W_{1,p}^{2D}[\omega]\leq n}\omega$. Note that $\omega(\{x:\mathbf{W}%
_{1,p}^{2D}[\omega](x)=\infty\})=0$ if and only if $\omega\in\mathcal{M}%
_{0,e}(\Omega)$, see \cite{Mi}.\medskip
\end{remark}

\begin{remark}
\label{exten}As in \cite{BiNQVe}, from Theorem \ref{main}, we can extend
Theorem \ref{main1} given for $\mathcal{G}(u)=\left\vert u\right\vert ^{q-1}%
u$, to the case of a function $\mathcal{G}(x,t,.),$ odd for $a.e.$ $(x,t)\in
Q,$ such that
\[
|\mathcal{G}(x,t,u)|\leq G(|u|),\qquad\int_{1}^{\infty}G(s)s^{-q-1}ds<\infty,
\]
where $G$ is a nondecreasing continuous, under the condition that $\omega$
does not charge the sets of zero $\mathrm{Cap}_{\mathbf{G}_{p,\frac{q}%
{q+1-p},1}}$-capacity, where for any Borel set $E\subset\mathbb{R}^{N},$
\[
\mathrm{Cap}_{\mathbf{G}_{p,\frac{q}{q+1-p},1}}(E)=\inf\{||\varphi
||_{L^{\frac{q}{q-p+1},1}(\mathbb{R}^{N})}:\varphi\in L^{\frac{q}{q-p+1}%
,1}(\mathbb{R}^{N}),\quad\mathbf{G}_{p}\ast\varphi\geq\chi_{E}\}
\]
where $L^{\frac{q}{q-p+1},1}(\mathbb{R}^{N})$ is the Lorentz space of order
$(q/(q-p+1),1)$.
\end{remark}

\subsection{The exponential case}

Theorem \ref{expo} extends the elliptic result of \cite[Theorem 1.2]{BiNQVe}
to the parabolic case. For the proof, we use the following property of
\cite[Theorem 2.4]{BiNQVe}:

\begin{proposition}
\label{majex}Suppose $1<p<N.$ Let $\nu\in\mathcal{M}_{b}^{+}(\Omega),$
$\beta>1,$ and $\delta_{0}=((12\beta)^{-1})^{\beta}p\ln2.$ There exists
$C=C(N,p,\beta,D)$ such that, for any $\delta\in\left(  0,\delta_{0}\right)
$,%
\[
\int_{\Omega}\exp(\delta\frac{(\mathbf{W}_{1,p}^{2D}[\nu])^{\beta}%
}{||\mathbf{M}_{p,2D}^{\frac{{p-1}}{\beta^{\prime}}}[\nu]|{|_{{L^{\infty}%
}({\mathbb{R}^{N}})}^{\frac{\beta}{p-1}}}})dx\leq\frac{C}{\delta_{0}-\delta}.
\]
\medskip
\end{proposition}

\begin{proof}
[Proof of Theorem \ref{expo}]Let $Q_{n}$ be defined at (\ref{qn}), and
$\omega_{n}=\omega\chi_{\Omega_{n}},$ where $\Omega_{n}=\{x\in\Omega
:d(x,\partial\Omega)>1/n\}.$ We still consider $\mu_{1},\mu_{2},F_{n}%
,\mu_{1,n},\mu_{2,n}$ as in (\ref{rs}), (\ref{sn}).\medskip

Case (i): Assume that $||F||_{L^{\infty}((0,T))}\leq1$ and (\ref{plou}) holds.
We have $\mu_{1,n},\mu_{2,n}\leq n\chi_{\Omega}+\omega$. For any
$\varepsilon>0,$ there exists $c_{\varepsilon}=c_{\varepsilon}(\varepsilon
,N,p,\beta,\kappa,$\noindent$D)$ $>0$ such that
\[
(n+\kappa\mathbf{W}_{1,p}^{2D}[n\chi_{\Omega}+\omega])^{\beta}\leq
c_{\varepsilon}n^{\frac{\beta p}{p-1}}+(1+\varepsilon)\kappa^{\beta
}(\mathbf{W}_{1,p}^{2D}[\omega])^{\beta}%
\]
$a.e.$ in $\Omega$. Thus,
\[
\exp\left(  \tau(n+\kappa\mathbf{W}_{1,p}^{2D}[n\chi_{\Omega}+\omega])^{\beta
}\right)  \leq\exp\left(  \tau c_{\varepsilon}n^{\frac{\beta p}{p-1}}\right)
\exp\left(  \tau(1+\varepsilon)\kappa^{\beta}(\mathbf{W}_{1,p}^{2D}%
[\omega])^{\beta}\right)  .
\]
If (\ref{plou}) holds with $M_{0}=\left(  \delta_{0}/\tau\kappa^{\beta
}\right)  ^{(p-1)/\beta}$ then we can chose $\varepsilon$ such that
\[
\tau(1+\varepsilon)\kappa^{\beta}||\mathbf{M}_{p,2D}^{\frac{{p-1}}%
{\beta^{\prime}}}[\nu]|{|_{{L^{\infty}}({\mathbb{R}^{N}})}^{\frac{\beta}{p-1}%
}<}\delta_{0}.
\]
From Proposition \ref{majex}, we get $\exp(\tau(1+\varepsilon)\kappa^{\beta
}\mathbf{W}_{1,p}^{2D}[\omega])^{\beta})\in L^{1}(\Omega),$ which implies
$\exp(\tau(n+\kappa^{\beta}\mathbf{W}_{1,p}^{2D}[n\chi_{\Omega}+\omega
])^{\beta})\in L^{1}(\Omega)$ for all $n$. We conclude from Theorem
\ref{main}.\medskip

\noindent Case (ii): Assume that there exists $\varepsilon>0$ such that
$\mathbf{M}_{p,2D}^{(p-1)/(\beta+\varepsilon)^{\prime}}[\omega]\in L^{\infty
}(\mathbb{R}^{N})$. Now we use the inequality $\mu_{1,n},\mu_{2,n}\leq
n(\chi_{\Omega}+\omega)$. For any $\varepsilon>0$ and any $n\in\mathbb{N}$
there exists $c_{\varepsilon,n}>0$ such that
\[
(n+\kappa\mathbf{W}_{1,p}^{2D}[n(\chi_{\Omega}+\omega)])^{\beta}\leq
c_{\varepsilon,n}+\varepsilon(\mathbf{W}_{1,p}^{2D}[\omega])^{\beta_{0}}.
\]
Thus, from Proposition \ref{majex}, we obtain that $\exp(\tau(n+\kappa^{\beta
}\mathbf{W}_{1,p}^{2D}[n(\chi_{\Omega}+\omega)])^{\beta})\in L^{1}(\Omega)$
for any $n\in\mathbb{N}$. We conclude from Theorem \ref{main}.
\end{proof}

\section{General case with source term\label{sour}}

\noindent The results of this Section are based on Corollary \ref{TH5} and
elliptic techniques of Wolff potential used in \cite{PhVe1}, \cite{PhVe2} and
\cite[Theorem 2.5]{NQVe}.

\subsection{The power case}

\noindent Recall some results of \cite{PhVe1}, \cite{PhVe2} for the
nonnegative solutions of equation
\begin{equation}
-\Delta_{p}u=u^{q}+\mathcal{\omega}\quad\text{in }\Omega,\qquad u=0\quad
\text{on }\partial\Omega. \label{mods}%
\end{equation}
It was proved that if $\omega(E)\leq C\mathrm{Cap}_{\mathbf{G}_{p,\frac
{q}{q+1-p}}}(E),$for any compact of $\mathbb{R}^{N},$ with $C$ small enough,
problem (\ref{mods}) has at least a solution, and conversely if there exists a
solution, and $\omega$ has a compact support, then there exists a constant
$C^{\prime}$ such that
\[
\omega(E)\leq C^{\prime}\mathrm{Cap}_{\mathbf{G}_{p,\frac{q}{q+1-p}}%
}(E),\qquad\text{for any compact set }E\text{ of }\mathbb{R}^{N}.
\]

For proving Theorem \ref{120410} we use the following property of Wolff
potentials, shown in \cite{PhVe1}:

\begin{theorem}
\label{12042}Let $q>p-1$, $0<p<N$, $\omega\in\mathcal{M}_{b}^{+}(\Omega)$. If
for some $\lambda>0,$
\begin{equation}
\omega(E)\leq\lambda\mathrm{Cap}_{\mathbf{G}_{p,\frac{q}{q+1-p}}}(E)\quad
\quad\text{for any compact set }E\subset\mathbb{R}^{N}, \label{cont}%
\end{equation}
then $(\mathbf{W}_{1,p}^{2\noindent D}[\omega])^{q}\in L^{1}(\Omega),$ and
there exists $M=M(N,p,q,\mathrm{diam}(\Omega))$ such that, $a.e.$ in
$\Omega,$
\begin{equation}
\mathbf{W}_{1,p}^{2D}\left[  (\mathbf{W}_{1,p}^{2D}[\omega])^{q}\right]  \leq
M\lambda^{\frac{q-p+1}{(p-1)^{2}}}\mathbf{W}_{1,p}^{2D}[\omega]<\infty.
\label{12044}%
\end{equation}

\end{theorem}

We deduce the following:

\begin{lemma}
\label{12047}Let $\omega\in$ $\mathcal{M}_{b}^{+}(\Omega)$, and $b\geq0$ and
$K>0$. Suppose that $\{u_{m}\}_{m\geq1}$ is a sequence of nonnegative
functions in $\Omega$ that satisfies
\begin{align*}
{u_{1}}  &  \leq K\mathbf{W}_{1,p}^{2D}[\omega]+b,\\
{u_{m+1}}  &  \leq K\mathbf{W}_{1,p}^{2D}[u_{m}^{q}+\omega]+b\qquad\forall
m\geq1.
\end{align*}
Assume that $\omega$ satisfies (\ref{cont}) for some $\lambda>0.$ Then there
exist $\lambda_{0}$ and $b_{0},$ depending on $N,p,q,K,${$D$}${,}$ such that,
if $\lambda\leq\lambda_{0}$ and $b\leq b_{0}$, then $\mathbf{W}_{1,p}%
^{2D}[\omega]\in L^{q}(\Omega)$ and for any $m\geq1,$
\begin{equation}
{u_{m}}\leq2\beta_{p}K\mathbf{W}_{1,p}^{2D}[\omega]+2b,\qquad\beta_{p}%
=\max({1,{3^{\frac{{2-p}}{{p-1}}})}}. \label{12043}%
\end{equation}

\end{lemma}

\begin{proof}
Clearly, (\ref{12043}) holds for $m=1$. Now, assume that it holds at the order
$m$. Then
\[
u_{m}^{q}\leq2^{q-1}{\left(  {2}\beta_{p}\right)  ^{q}}{K^{q}(}{\mathbf{W}%
{_{1,p}^{2D}[\omega]})^{q}}+2^{q-1}{\left(  {2b}\right)  ^{q}.}%
\]
Using (\ref{12044}) we get
\begin{align*}
u_{m+1}  &  \leq K\mathbf{W}_{1,p}^{2D}\left[  {2^{q-1}{{\left(  {2}\beta
_{p}\right)  }^{q}}K^{q}({{{W_{1,p}^{2D}[\omega])}}^{q}}+2^{q-1}{{\left(
{2b}\right)  }^{q}}+\omega}\right]  +b\\
&  \leq\beta_{p}K\left(  {{A_{1}}\mathbf{W}_{1,p}^{2D}\left[  ({{{{W_{1,p}%
^{2D}[\omega])}}^{q}}}\right]  +\mathbf{W}_{1,p}^{2D}\left[  {{{\left(
{2b}\right)  }^{q}}}\right]  +W_{1,p}^{2D}[\omega]}\right)  +b\\
&  \leq\beta_{p}K(A_{1}M\lambda^{\frac{{q-p+1}}{{{{(p-1)}^{2}}}}}%
+1)\mathbf{W}_{1,p}^{2D}[\omega]+\beta_{p}K\mathbf{W}_{1,p}^{2D}\left[
{{{\left(  {2b}\right)  }^{q}}}\right]  +b\\
&  =\beta_{p}K(A_{1}M\lambda^{\frac{{q-p+1}}{{{{(p-1)}^{2}}}}}+1)\mathbf{W}%
_{1,p}^{2D}[\omega]+A_{2}b^{\frac{q}{p-1}}+b,
\end{align*}
where $M$ is as in (\ref{12044}) and
\[
A_{1}=\left(  2^{q-1}{{\left(  {2}\beta_{p}\right)  }^{q}}K^{q}\right)
^{1/(p-1)},\qquad A_{2}=\beta_{p}K2^{q/(p-1)}|{B_{1}}{|}^{1/(p-1)}{(}%
p^{\prime})^{-1}{\left(  {2D}\right)  ^{p^{\prime}}}.
\]
\newline Thus, (\ref{12043}) holds for $m=n+1$ if we prove that
\[
A_{1}M\lambda^{\frac{{q-p+1}}{{{{(p-1)}^{2}}}}}\leq1\text{ and }A_{2}%
b^{\frac{q}{p-1}}\leq b,
\]
which is equivalent to
\[
\lambda\leq(A_{1}M)^{-\frac{(p-1)^{2}}{q-p+1}}\text{ and }b\leq A_{2}%
^{-\frac{p-1}{q-p+1}}.
\]
Therefore, we obtain the result with $\lambda_{0}=(A_{1}M)^{-(p-1)^{2}%
/(q-p+1)}$ and $b_{0}=A_{2}^{-(p-1)/(q-p+1)}.$\medskip
\end{proof}

\begin{proof}
[Proof of Theorem \ref{120410}]From Corollary \ref{051120131} and \ref{TH5},
we can construct a sequence of nonnegative \textit{nondecreasing} R-solutions
$\{u_{m}\}_{m\geq1},$ defined in the following way: $u_{1}$ is a R-solution of
(\ref{pmu}), and $u_{m+1}$ is a nonnegative R-solution of
\[
\left\{
\begin{array}
[c]{l}%
(u_{m+1})_{t}-\mathcal{A}(u_{m+1})=u_{m}^{q}+\mu\qquad\text{in }Q,\\
u_{m+1}=0\qquad\text{on }\partial\Omega\times(0,T),\\
u_{m+1}(0)=u_{0}\qquad\text{in }\Omega.
\end{array}
\right.
\]
Setting $\overline{u}_{m}=\sup_{t\in(0,T)}u_{m}(t)$ for all $m\geq1,$ there
holds
\begin{align*}
{\overline{u}_{1}}  &  \leq\kappa\mathbf{W}_{1,p}^{2D}[\omega]+||u_{0}%
||_{L^{\infty}(\Omega)},\\
{\overline{u}_{m+1}}  &  \leq\kappa\mathbf{W}_{1,p}^{2D}[\overline{u}_{m}%
^{q}+\omega]+||u_{0}||_{L^{\infty}(\Omega)}\quad\quad\forall m\geq1.
\end{align*}
From Lemma \ref{12047}, we can find $\lambda_{0}=\lambda_{0}(N,p,q,D)$ and
$b_{0}=b_{0}(N,p,q,D)$ such that if (\ref{051120132}) is satisfied with
$\lambda_{0}$ and $b_{0}$; then
\begin{equation}
u_{m}\leq{\overline{u}_{m}}\leq2\beta_{p}\kappa\mathbf{W}_{1,p}^{2D}%
[\omega]+2||u_{0}||_{L^{\infty}(\Omega)}\quad\quad\forall m\geq1.
\label{12049}%
\end{equation}
Thus $\left\{  u_{m}\right\}  $ converges $a.e.$ in $Q$ and in $L^{q}(Q)$ to
some function $u,$ for which (\ref{maw}) is satisfied in $\Omega$ with
$c=2\beta_{p}\kappa.$ Finally, one can apply Theorem \ref{sta} to the sequence
of measures $\left\{  u_{m}^{q}+\mu\right\}  ,$ and obtain that $u$ is a
R-solution of (\ref{pro3}).
\end{proof}

\subsection{The exponential case}

We end this Section by proving Theorem \ref{MTH1}. We first recall an
approximation property, which is a consequence of \cite[Theorem 2.5]{NQVe}:

\begin{theorem}
\label{TH3} Let $\tau>0$, $b\geq0$, $K>0$, $l\in\mathbb{N}$ and $\beta\geq1$
such that $l\beta>p-1$. Let $\mathcal{E}$ be defined by (\ref{ess}). Let
$\{v_{m}\}$ be a sequence of nonnegative functions in $\Omega$ such that, for
some $K>0,$
\begin{align*}
v{_{1}}  &  \leq K\mathbf{W}_{1,p}^{2D}[\mu]+b,\\
v{_{m+1}}  &  \leq K\mathbf{W}_{1,p}^{2D}[\mathcal{E}(\tau v_{m}^{\beta}%
)+\mu]+b,\quad\forall m\geq1.
\end{align*}
Then, there exist $b_{0}$ and $M_{0},$ depending on $N,p,\beta,\tau,l,K,D,$
such that if $b\leq b_{0}$ and
\begin{equation}
||\mathbf{M}_{p,{2D}}^{\frac{{(p-1)(\beta-1)}}{\beta}}[\mu]|{|_{\infty
,{\mathbb{R}^{N}}}}\leq M_{0}, \label{mai}%
\end{equation}
then, setting $c_{p}=2{{{{{\max({1,}2{{^{\frac{{2-p}}{{p-1}}}),}}}}}}}$%
\[
{\exp(\tau{{{({Kc_{p}\mathbf{W}_{1,p}^{2D}[\mu]+2{b_{0})}}}^{\beta}})}\in
L^{1}(\Omega),}%
\]%
\begin{equation}
{v_{m}}\leq Kc_{p}W_{1,p}^{2D}[\mu]+2{b_{0}},\quad\forall m\geq1.\,
\label{1232}%
\end{equation}

\end{theorem}

\begin{proof}
[Proof of Theorem \ref{MTH1}]From Corollary \ref{051120131} and \ref{TH5} we
can construct a sequence of nonnegative\textit{ nondecreasing} R-solutions
$\{u_{m}\}_{m\geq1}$ defined in the following way: $u_{1}$ is a R-solution of
problem (\ref{pmu}), and by induction, $u_{m+1}$ is a R-solution of
\begin{equation}
\left\{
\begin{array}
[c]{l}%
(u_{m+1})_{t}-\mathcal{A}(u_{m+1})=\mathcal{E}(\tau u_{m}^{\beta})+\mu
\qquad\text{in }Q,\\
u_{m+1}=0\qquad\text{on }\partial\Omega\times(0,T),\\
u_{m+1}(0)=u_{0}\qquad\text{in }\Omega.
\end{array}
\right.  \label{arc}%
\end{equation}
And, setting $\overline{u}_{m}=\sup_{t\in(0,T)}u_{m}(t),$ there holds
\begin{align*}
{\overline{u}_{1}}  &  \leq{\kappa}W_{1,p}^{2D}[\omega]+||u_{0}||_{\infty
,\Omega},\\
\qquad{\overline{u}_{m+1}}  &  \leq\kappa W_{1,p}^{2D}[\mathcal{E}%
(\tau\overline{u}_{m}^{\beta})+\omega]+||u_{0}||_{L^{\infty}(\Omega)}%
,\quad\quad\forall m\geq1.
\end{align*}
Thus, from Theorem \ref{TH3}, there exist $b_{0}\in(0,1]$ and $M_{0}>0,$
depending on $N,p,\beta,\tau,l,D,$ such that, if (\ref{mai}) holds, then
(\ref{1232}) is satisfied with $v_{m}=\overline{u}_{m}$. As a consequence,
$u_{m}$ is well defined. Thus, $\left\{  u_{m}\right\}  $ converges $a.e.$ in
$Q$ to some function $u,$ for which (\ref{1334}) is satisfied in $\Omega$.
Furthermore, $\left\{  \mathcal{E}(\tau u_{m}^{\beta})\right\}  $ converges to
$\mathcal{E}(\tau u^{\beta})$ in $L^{1}(Q)$. Finally, one can apply Theorem
\ref{sta} to the sequence of measures $\left\{  \mathcal{E}(\tau u_{m}^{\beta
})+\mu\right\}  ,$ and obtain that $u$ is a R-solution of (\ref{pro2}%
).\medskip
\end{proof}

\begin{remark}
In \cite[Theorem 1.1]{NQVe}, when $\mathcal{A}=\Delta_{p}$, we showed that
there exist $M=M(N,p,\beta,\tau,l,D)$ such that if
\[
||\mathbf{M}_{p,2D}^{\frac{{(p-1)(\beta-1)}}{\beta}}[\omega]|{|_{{L^{\infty}%
}({\mathbb{R}^{N}})}}\leq M,
\]
then the problem
\begin{equation}
\left\{
\begin{array}
[c]{l}%
-\Delta_{p}v=\mathcal{E}(\tau v^{\beta})+\omega\qquad\text{in }\Omega,\\
v=0\qquad\text{on }\partial\Omega.
\end{array}
\right.  \label{pas}%
\end{equation}
has a renormalized solution in the sense of \cite{DMOP}. We claim the
following: \medskip

Let $\mathcal{A}=\Delta_{p}$ and $u_{0}\equiv0.$ If (\ref{pas}) has a
renormalized solution $v$ and $\omega\in\mathcal{M}_{0,e}(\Omega)$, then the
problem (\ref{pro2}) in Theorem \ref{MTH1} admits a R-solution $u$, satisfying
$u(x,t)\leq v(x)$ a.e in $Q$.\medskip

\noindent Indeed, since $\omega\in\mathcal{M}_{0,e}(\Omega)$, there holds
$\mu\in\mathcal{M}_{0}(Q)$. Otherwise, for any measure $\eta\in\mathcal{M}%
_{0}(Q)$ the problem
\[
\left\{
\begin{array}
[c]{l}%
u_{t}-\Delta_{p}u=\eta\qquad\text{in }Q,\\
u=0\qquad\text{on }\partial\Omega\times(0,T),\\
u=0\qquad\text{in }\Omega,
\end{array}
\right.
\]
has a (unique) R-solution, and the comparison principle is valid, see
\cite{PePoPor}. Thus, as in the proof of Theorem \ref{MTH1}, we can construct
a \textbf{unique} sequence of nonnegative nondecreasing R-solutions
$\{u_{m}\}_{m\geq1},$ defined in the following way: $u_{1}$ is a R-solution of
problem (\ref{pmu}) and satisfies $u_{1}\leq v$ a.e in $Q$ ; and by induction,
$u_{m+1}$ is a R-solution of (\ref{arc}) and satisfies $u_{m+1}\leq v$ a.e in
$Q$. Then $\left\{  \mathcal{E}(\tau u_{m}^{\beta})\right\}  $ converges to
$\mathcal{E}(\tau u^{\beta})$ in $L^{1}(Q)$. Finally, $u:=\lim_{n\rightarrow
\infty}{u_{n}}$ is a solution of (\ref{pro2}). Clearly, this claim is also
valid for power source term.
\end{remark}

\end{document}